\documentclass[a4paper,12pt,twoside]{article}      
\usepackage{amsmath,amssymb,amsfonts,amsthm} 
\usepackage{graphics}                 
\usepackage{color}                    
\usepackage{hyperref}                 
\usepackage[all]{xy}

\newtheorem{theorem}{Theorem}[section]
\newtheorem{lemma}[theorem]{Lemma}
\newtheorem{proposition}[theorem]{Proposition}

\newtheorem{corollary}[theorem]{Corollary}
\newtheorem{definition}[theorem]{Definition}
\newtheorem{example}[theorem]{Example}
\newtheorem{remark}[theorem]{Remark}
\numberwithin{equation}{section}

\def\rad{\mathop{\hbox{rad}}}

\title{A Characterization of Finite EI Categories with Hereditary Category Algebras}
\author{\footnotesize{Liping Li} \\
\footnotesize{School of Mathematics, University of Minnesota, MN, 55455, USA}}

\begin{document}
\maketitle

\begin{abstract}
In this paper we give an explicit algorithm to construct the ordinary quiver of a finite EI category for which the endomorphism groups of all objects have orders invertible in the field $k$. We classify all finite EI categories with hereditary category algebras, characterizing them as free EI categories (in a sense which we define) for which all endomorphism groups of objects have invertible orders. Some applications on the representation types of finite EI categories are derived.
\end{abstract}

\section{Introduction}

In this paper we study the representations of \textit{finite EI categories}. They are small categories with finitely many morphisms in which every endomorphism is an isomorphism. This concept includes many structures such as finite groups, finite posets and free categories associated to finite quivers. A \textit{representation} of $\mathcal{C}$ is a covariant functor from $\mathcal{C}$ to the category of finite dimensional vector spaces, and it may be regarded as the same thing as a finitely generated module for the category algebra.\

We introduce the concepts of \textit{finite EI quivers} and \textit{finite free EI categories}. A finite EI quiver is a quiver equipped with two maps $f$ and $g$. The map $f$ assigns a finite group $f(v) = G_v$ to each vertex $v$ in the quiver and the map $g$ assigns a $(G_w, G_v)$-biset to each arrow from the vertex $v$ to the vertex $w$ in the quiver. Finite free EI categories are categories generated from finite EI quivers by a specific rule which generalizes the construction of a free category from a quiver ~\cite{MacLane}. We will show that they are characterized by a certain unique factorization property of the non-isomorphisms. In a certain sense a finite free EI category is the largest EI category which can be generated by a finite EI quiver.\

For every finite EI category $\mathcal{C}$, we will construct a finite free EI category $\hat{\mathcal{C}}$ and a full functor $\hat{F}: \hat{\mathcal{C}} \rightarrow \mathcal{C}$ such that $\hat{F}$ is the identity map restricted to objects, endomorphisms and all non-isomorphisms which can not be expressed as composites of more than one non-isomorphism. This category $\hat{\mathcal{C}}$  is unique up to isomorphism. We call $\hat{\mathcal{C}}$ the \textit{free EI cover} of $\mathcal{C}$.\

A finite EI category $\mathcal{C}$ determines a finite dimensional associative $k$-algebra $k\mathcal{C}$ with identity, called the category algebra (defined in section 2). Under the assumption that the endomorphism groups of all objects in $\mathcal{C}$ have orders invertible in $k$, we will describe an explicit algorithm to construct the ordinary quiver of the category algebra $k\mathcal{C}$. This algorithm uses the representations of semisimple group algebras and their restrictions to subgroups. It is easier and more intuitive than the method by computing Exts, introduced in ~\cite{Simson}. Our main results are described in the following theorems:

\begin{theorem}
Let $\mathcal{C}$ be a finite EI category for which the endomorphism groups of all objects have orders invertible in the field $k$. Then the quiver $Q$ constructed by our algorithm is precisely the ordinary quiver of the category algebra $k\mathcal{C}$. Moreover, $k\mathcal{C}$ has the same ordinary quiver as that of $k \hat{\mathcal{C}}$, the category algebra of the free EI cover $\hat{\mathcal{C}}$ of $\mathcal{C}$.
\end{theorem}

\begin{theorem}
Let $\mathcal{C}$ be a finite EI category. Then the category algebra $k\mathcal{C}$ is hereditary if and only if $\mathcal{C}$ is a finite free EI category satisfying the condition that the endomorphism groups of all objects have orders invertible in $k$.
\end{theorem}

These results are useful in determining the representation types of category algebras. Indeed, if $\mathcal{C}$ is a finite free EI category for which the endomorphism groups of all objects have orders invertible in $k$, then it has the same representation type as that of its ordinary quiver. Stated explicitly, $k\mathcal{C}$ is of finite (tame, resp.) representation type if and only if the underlying graph of its ordinary quiver is a disjoint union of Dynkin (Euclidean, resp.) diagrams. Otherwise, it is of wild type. For a general finite EI category $\mathcal{C}$, $k\mathcal{C}$ is a quotient algebra of $k\hat {\mathcal{C}}$, the category algebra of its free EI cover $\hat {\mathcal{C}}$. Thus the finite representation type of $k \hat {\mathcal{C}}$ implies the finite representation type of $k\mathcal{C}$, no matter what the characteristic of $k$ is. Moreover, if $\mathcal{C}$ has a full subcategory of infinite representation type, it is of infinite representation type as well. By studying those full subcategories which are finite free EI categories, for example the connected full subcategories with two objects, we can get certain information about the representation types of general finite EI categories.\

Here is the layout of this paper: in section 2 we give some background on the representation theory of finite EI categories, and introduce the definitions of finite EI quivers and finite free EI categories. An equivalent definition by the unique factorization property of non-isomorphisms is proved, and we present some further basic properties.\

From section 3 onwards we focus on the representations of finite EI categories for which all endomorphism groups of objects have orders invertible in $k$. In section 4 we give a detailed algorithm to construct an associated quiver $Q$ for each finite EI category for which the endomorphism groups of all objects have orders invertible in the field $k$. Then we prove Theorem 1.1. In section 5 we prove Theorem 1.2, which classifies all finite EI categories with hereditary category algebras.\

Even though we can study the representation type of $\mathcal{C}$ by constructing its ordinary quiver, it is more efficient sometimes to use certain simple criteria deduced from Theorem 1.1 and Theorem 1.2. In the last section we describe such criteria expressed in terms of full subcategories with two objects.\

It is sometimes useful to know an explicit functor $F$ from $k\mathcal{C}$-mod, the category of all representations of a finite free EI category $\mathcal{C}$, to $kQ$-mod, the category of all representations of the ordinary quiver $Q$ of $k\mathcal{C}$. Although our proof of Theorem 1.2 does not rely on this functor, we describe it in detail in the appendix. This functor is proved to be full, faithful and dense, and hence induces a Morita-equivalence between $k\mathcal{C}$ and $kQ$.\

In this paper $\mathcal{C}$ is always a finite EI category with objects Ob$(\mathcal{C})$ and morphisms Mor$(\mathcal{C})$. We use Aut$_{\mathcal{C}}(x)$ (or End$_ {\mathcal{C}} (x)$) to denote the endomorphism group of a fixed object $x$ in $\mathcal{C}$ and Hom$_{\mathcal{C}}(x,y)$ to denote the set of morphisms from an object $x$ to another object $y$. The field $k$ is supposed to be algebraically closed. All modules are finitely generated left modules. Composition of group actions, morphisms and maps is from right to left.

\textbf{Acknowledgement:} The author wants to express great appreciation to his advisor, Professor Peter Webb, who led the author into this area. Without his invaluable suggestions and help this paper would not have appeared. Professor Webb also contributed the concepts of finite EI quivers and finite free EI categories for this paper. These ideas are described in his unpublished notes.

\section{Finite Free EI Categories}
For the reader's convenience, we include in this section some background on the representation theory of finite EI categories. Please refer to ~\cite{Webb1}, ~\cite{Xu1}, ~\cite{Xu2} for more details.\

The category algebra $k\mathcal{C}$ of a finite EI category $\mathcal{C}$ is the $k$-space with basis the morphisms in $\mathcal{C}$. Its product $*$ is determined by bilinearity from the product of the basis elements, given by the rule:

\begin{equation*}
\alpha * \beta = \left \{\begin{array} {rl}
\alpha \circ \beta &\text{if } \alpha \text{ and } \beta \text{ can
be composed} \\
0 &\text{otherwise.}
\end{array} \right.
\end{equation*}
Let $R$ be representation of $\mathcal{C}$ defined in section 1. Then it assigns a vector space $R(x)$ to each object $x$ in $\mathcal{C}$, and a linear transformation $R(\alpha): R(x) \rightarrow R(y)$ to each morphism $\alpha: x \rightarrow y$ such that all composition relations of morphisms in $\mathcal{C}$ are preserved under $R$. Notice that $R(x)$ has a $k\text{End}_{\mathcal{C}}(x)$-module structure for every object $x$ in $\mathcal{C}$. Indeed, for each morphism $g$ in End$_{\mathcal{C}}(x)$, $R(g)$ is an automorphism of $R(x)$. Thus we can define an action of $g$ on $R(x)$ by letting $g \cdot v = R(g)(v)$, for all $v$ in $R(x)$. A \textit{homomorphism} $\varphi: R_1 \rightarrow R_2$ of two representations is a natural transformation of functors. By Theorem 7.1 of ~\cite{Mitchell}, a representation of $\mathcal{C}$ is equivalent to a $k\mathcal{C}$-module. Thus we don't distinguish these two concepts throughout this paper.\

A finite EI category $\mathcal{C}$ is said to be of \textit{finite (tame, wild}, resp.) \textit{representation type} if the category algebra $k\mathcal{C}$ is of finite (tame, wild, resp.) type. The category $\mathcal{C}$ is \textit{connected} if for any two distinct objects $x$ and $y$, there is a list of objects $x=x_0, x_1, \ldots, x_n=y$ such that either Hom$_{\mathcal{C}} (x_i, x_{i+1})$ or Hom$_{\mathcal{C}} (x_{i+1}, x_i)$ is not empty, $0 \leqslant i \leqslant n-1$. Every finite EI category is a disjoint union of connected components, and each component is a full subcategory. If $\mathcal{C} = \bigsqcup _{i=1}^{m} \mathcal{C}_i$, the category algebra $k\mathcal{C}$ has an algebra decomposition $k\mathcal{C}_1 \oplus \ldots \oplus k\mathcal{C}_m$. Moreover, if $\mathcal{C}$ and $\mathcal{D}$ are equivalent finite EI categories, $k \mathcal{C}$-mod is equivalent to $k \mathcal{D}$-mod by Proposition 2.2 in ~\cite{Webb1}. Thus it is sufficient to study the representations of connected, skeletal finite EI categories. We make the following convention:\\

\noindent \textbf{Convention:} \textit{All finite EI categories in this paper are \textbf{connected} and \textbf{skeletal}. Thus endomorphisms, isomorphisms and automorphisms coincide.}\\

Under the hypothesis that $\mathcal{C}$ is skeletal, if $x$ and $y$ are two distinct objects in $\mathcal{C}$ with Hom$_{\mathcal{C}} (x,y)$ non-empty, then Hom$_{\mathcal{C}} (y,x)$ is empty. Indeed, if this is not true, we can take $\alpha \in \text{Hom}_{\mathcal{C}} (y, x)$ and $\beta \in \text{Hom}_{\mathcal{C}} (x,y)$. The composite $\beta \alpha$ is an endomorphism of $y$, hence an automorphism. Similarly, the composite $\alpha \beta$ is an automorphism of $x$. Thus both $\alpha$ and $\beta$ are isomorphisms, so $x$ is isomorphic to $y$. But this is impossible since $\mathcal{C}$ is skeletal and $x \neq y$.\

It is time to introduce the concept of \textit{finite free EI categories}. Before giving a formal definition, we define \textit{finite EI quivers}, which are finite quivers with extra structure.

\begin{definition}
A finite EI quiver $\hat{Q}$ is a datum $(Q_0, Q_1, s, t, f, g)$, where: $(Q_0, Q_1, s, t)$ is a finite acyclic quiver with vertex set $Q_0$, arrow set $Q_1$, source map $s$ and target map $t$. The map $f$ assigns a finite group $f(v)$ to each vertex $v \in Q_0$; the map $g$ assigns an $(f(t(\alpha)), f(s(\alpha)))$-biset to each arrow $\alpha \in Q_1$.
\end{definition}

If $f$ assigns the trivial group to each vertex in $Q_0$ in the above definition, we obtain a quiver in the usual sense. In this sense, finite acyclic quivers are special cases of finite EI quivers.\

Each finite EI quiver $\hat{Q} = (Q_0, Q_1, s, t, f, g)$ determines a finite EI category $\mathcal{C_{\hat{Q}}}$ in the following way: the objects in $\mathcal{C_{\hat{Q}}}$ are precisely the vertices in $Q_0$. For a particular object $v$ in $\mathcal{C_{\hat{Q}}}$, we define End$_{\mathcal{C_{\hat{Q}}}} (v) = f (v)$, which is a finite group by our definition. It remains to define Hom$_{\mathcal{C_{\hat{Q}}}} (v,w)$ if $v \neq w$ are distinct vertices in $Q_0$, and the composition of morphisms.\

Let $\xymatrix{v \ar@{~>}[r]^{\gamma} & w}$ be a directed path from $v$ to $w$. Then $\gamma$ can be written uniquely as a composition of arrows, where $v_i \in Q_0$ and $\alpha_i \in Q_1$ for $i=1, \ldots, n$.
\begin{equation*}
\xymatrix{v=v_0 \ar[r]^{\alpha_1} & v_1 \ar[r]^{\alpha_2} & \ldots \ar[r]^{\alpha_n} & v_n=w}
\end{equation*}
Notice that $g(\alpha_i)$ is an $(f(v_i), f(v_{i-1}))$-biset, so we define:
\begin{equation*}
H_{\gamma}=g(\alpha_n) \times_{f(v_{n-1})} g(\alpha_{n-1}) \times_{f(v_{n-2})} \ldots \times_{f(v_1)} g(\alpha_1),
\end{equation*}
the biset product defined in ~\cite{Webb2}. Finally, Hom$_{\mathcal{C_{\hat{Q}}}} (f(v),f(w))$ can be defined as $\bigsqcup _{\gamma} H_{\gamma}$, the disjoint union of all $H_{\gamma}$, over all possible paths $\gamma$ from $x$ to $y$. In the case $v = w$ we define Hom$_{\mathcal{C_{\hat{Q}}}} (v,v) = f(v)$.\

Let $\alpha$ and $\beta$ be two morphisms in $\mathcal{C_{\hat{Q}}}$. They lie in two sets $H_{\gamma_1}$ and $H_{\gamma_2}$, where $\gamma_1$ and $\gamma_2$ are two paths determined by $\alpha$ and $\beta$ respectively, possibly of length 0. Their composite $\beta \circ \alpha$ can be defined by the following rule:\ it is 0 if the composite $\gamma_2 \gamma_1$ is not defined in $\hat{Q}$. Otherwise, the initial vertex $v$ of $\gamma_2$ is exactly the terminal vertex of $\gamma_1$. Since there is a natural surjective map $p: H_{\gamma_2} \times H_{\gamma_1} \rightarrow H_{\gamma_2} \times_{f(v)} H_{\gamma_1}$, we define $\beta \circ \alpha = p(\beta, \alpha)$, the image of $(\beta, \alpha)$ in $H_{\gamma_2} \times_{f(v)} H_{\gamma_1}$. This definition satisfies the associative rule,
and in this way we get a finite EI category $\mathcal{C_{\hat{Q}}}$ from $\hat{Q}$.

\begin{definition}
A finite EI category $\mathcal{C}$ is a finite free EI category if it is isomorphic to the finite EI category $\mathcal{C_{\hat{Q}}}$ generated from a finite EI quiver $\hat{Q}$ by the above construction.
\end{definition}

In practice it is inconvenient to check whether a finite EI category $\mathcal{C}$ is free or not by using the definition. Fortunately, there is an equivalent characterization built upon unfactorizable morphisms: the Unique Factorization Property (UFP).

\begin{definition}
A morphism $\alpha: x \rightarrow z$ in $\mathcal{C}$ is unfactorizable if $\alpha$ is not an isomorphism and whenever it has a factorization as a composite $\xymatrix{x \ar[r]^{\beta} & y\ar[r]^{\gamma} & z}$, then either $\beta$ or $\gamma$ is an isomorphism.
\end{definition}

The reader may want to know the relation between the terminology \textit{unfactorizable morphism} and the term \textit{irreducible morphism} which is widely accepted and used, for example in ~\cite{Bautista} and ~\cite{Xu1}.\ Indeed, in this paper they coincide since we only deal with finite EI categories. But in a more general context, they are different, as we explain in the following example:

\begin{example}
Consider the following category $\mathcal{C}$ with two objects $x \ncong y$. The non-identity morphisms in $\mathcal{C}$ are generated by $\alpha: x \rightarrow y$ and $\beta: y \rightarrow x$ with the only nontrivial relation being $\beta \alpha = 1_x$. Then the morphisms in $\mathcal{C}$ are $1_x, 1_y, \alpha, \beta \text{ and } \alpha\beta$. It is not a finite EI category since $\alpha \beta \in \text{End} _{\mathcal{C}} (y)$ is not an isomorphism. Then neither $\alpha$ nor $\beta$ are irreducible morphisms since one of them is a split monomorphism and the other is a split epimorphism. However, the reader can check that they are unfactorizable morphisms. Furthermore, the algorithm constructing ordinary quivers of this category algebra (described in section 4) still works well. This example illustrates the reason that we introduce the notion \textit{unfactorizable morphisms}: our algorithm works in a more general situation where unfactorizable morphisms do not coincide with irreducible morphisms any more.
\end{example}

\begin{equation*}
\xymatrix{ x \ar@(ul,dl)_{1_x} \ar@/^/[rr] ^{\alpha} & & y \ar@/^/[ll] ^{\beta} \ar@(ur,dr) ^{1_y}}
\end{equation*}

Note that the composite of an unfactorizable morphism with an isomorphism is still unfactorizable.

\begin{proposition}
Let $\alpha: x \rightarrow y$ be an unfactorizable morphism. Then $h\alpha g$ is also unfactorizable for every $h \in \text{Aut}_{\mathcal{C}}(y)$ and every $g \in \text{Aut}_{\mathcal{C}}(x)$.
\end{proposition}

\begin{proof}
Fix a decomposition $h\alpha g=\alpha_1 \alpha_2$. Then we have $\alpha = (h^{-1} \alpha_1) (\alpha_2 g^{-1})$. But $\alpha$ is unfactorizable, so by definition either one of $h^{-1}\alpha_1$ and $\alpha_2 g^{-1}$ is an isomorphism. Without loss of generality, let $h^{-1}\alpha_1$ be an isomorphism. Then $\alpha_1$ is an isomorphism since $h^{-1}$ is an automorphism.
\end{proof}

Let $\mathcal{C}$ be a finite EI category. By the previous proposition, the set of unfactorizable morphisms from an object $x$ to another object $y$ is closed under the actions of $\text{Aut}_{\mathcal{C}}(x)$ and $\text{Aut}_{\mathcal{C}}(y)$. Choose a fixed representative for each (Aut$_{\mathcal{C}} (y)$,Aut$_{\mathcal{C}} (x)$)-orbit. Repeating this process for all pairs of different objects $(x,y)$, we get a set $A =\{ \alpha_1, \ldots, \alpha_n\}$ of orbit representatives. Elements in $A$ are called \textit{representative unfactorizable morphisms}.\

We should point out here that each finite EI category $\mathcal{C}$ determines a finite EI quiver $\hat{Q}$ in the following way: its vertices are objects in $\mathcal{C}$; we put an arrow $x: \rightarrow y$ in $\hat{Q}$ for each representative unfactorizable morphism $\alpha: x \rightarrow y$ in $\mathcal{C}$. Thus the arrows biject with all representative unfactorizable morphisms $\alpha: x \rightarrow y$ in $\mathcal{C}$, or equivalently, all Aut$ _{\mathcal{C}} (y) \times \text{Aut} _{\mathcal{C}} (x)$-orbits of unfactorizable morphisms in $\mathcal{C}$. The map $f$ assigns End$_{\mathcal{C}} (x)$ to each object $x$; the map $g$ assigns the (Aut$_{\mathcal{C}} (y)$, Aut$_{\mathcal{C}} (x)$)-biset where a representative unfactorizable morphism $\alpha: x \rightarrow y$ lies to the corresponding arrow. Obviously, this finite EI quiver is unique up to isomorphism. We call this quiver the \textit{finite EI quiver} of $\mathcal{C}$.\

Now suppose that $\mathcal{C}$ is a finite free EI category. It is possible that there is more than one finite EI quiver generating $\mathcal{C}$, although they must have the same vertices. However, it is not hard to see that all those finite EI quivers are subquivers of the finite EI quiver of $\mathcal{C}$.\

All non-isomorphisms can be written as composites of unfactorizable morphisms.
\begin{proposition}
Let $\alpha: x \rightarrow y$ be a morphism with $x \neq y$. Then it has a decomposition \xymatrix{x=x_0 \ar[r]^{\alpha_1} & x_1 \ar[r]^{\alpha_2} & \ldots \ar[r]^{\alpha_n} & x_n=y}, where all $\alpha_i$ are unfactorizable.
\end{proposition}

\begin{proof}
If $\alpha$ is unfactorizable, we are done. Otherwise, $\alpha$ has a decomposition \xymatrix{x \ar[r]^{\alpha_1} & x_1 \ar[r]^{\alpha_2} & y} where neither $\alpha_1$ nor $\alpha_2$ is an isomorphism. In particular, $x_1$ is different from $x$ and $y$. If both $\alpha_1$ and $\alpha_2$ are unfactorizable, we are done. Otherwise, assume $\alpha_1$ is not unfactorizable. Repeating the above process, we can get a decomposition \xymatrix{x \ar[r] ^{\alpha_{11}} & x_{11} \ar[r]^{\alpha_{12}} & x_1 \ar[r]^{\alpha_2} & y}. With the same reasoning, $x, x_{11}, x_1, y$ are pairwise different. Since there are only finitely many objects, this process ends after finitely many steps. Therefore we get a decomposition of $\alpha$ into unfactorizable morphisms.
\end{proof}

For an arbitrary finite EI category $\mathcal{C}$, the ways to decompose a non-isomorphism into unfactorizable morphisms need not to be unique. However, we can show that for finite free EI categories, this decomposition is unique up to a trivial relation, i.e., they satisfy the property defined below:
\begin{definition}
A finite EI category $\mathcal{C}$ satisfies the Unique Factorization Property (UFP) if whenever any non-isomorphism $\alpha$ has two decompositions into unfactorizable morphisms:
\begin{equation*}
\xymatrix{x=x_0 \ar[r]^{\alpha_1} & x_1 \ar[r]^{\alpha_2} & \ldots \ar[r]^{\alpha_m} &x_m=y \\
x=x_0 \ar[r]^{\beta_1} & y_1 \ar[r]^{\beta_2} & \ldots \ar[r]^{\beta_n} & y_n=y}
\end{equation*}
then $m=n$, $x_i = y_i$, and there are $ h_i \in \text{Aut}_{\mathcal{C}} (x_i)$, $1 \leqslant i \leqslant n-1$ such that $\beta_1 = h_1 \alpha_1$, $\beta_2 =h_2 \alpha_2 h_1^{-1}$, \ldots, $\beta_{n-1} = h_{n-1} \alpha_{n-1} h_{n-2}^{-1}$, $\beta_n =\alpha_n h_{n-1}^{-1}$.
\end{definition}

The UFP gives a characterization of finite free EI categories.
\begin{proposition}
A finite EI category $\mathcal{C}$ is free if and only if it satisfies the UFP.
\end{proposition}

\begin{proof}
Suppose $\mathcal{C}$ is a finite free EI category generated by a finite EI quiver $\hat{Q}= (Q_0, Q_1, s, t, f, g)$. Let $\alpha: v \rightarrow w$ be an arbitrary non-isomorphism. By the previous proposition $\alpha$ can be written as a composite of unfactorizable morphisms. Let $\alpha_m \circ \ldots \circ \alpha_1$ and $\beta_n \circ \ldots \circ \beta_1$ be two such decompositions of $\alpha$. It is easy to see from definitions that an unfactorizable morphism in $\mathcal{C}$ lies in $g(\tau)$ for some unique arrow $\tau \in Q_1$. Thus $\alpha_m \circ \ldots \circ \alpha_1$ and $\beta_n \circ \ldots \circ \beta_1$ determine two paths $\gamma_1$ and $\gamma_2$ in $\hat{Q}$ from $v$ to $w$. But $\alpha$ is contained in Hom$_{\mathcal{C}} (v,w) = \bigsqcup _{\gamma} H_{\gamma}$, the disjoint union taken over all possible paths from $v$ to $w$, so $\gamma_1$ must be the same as $\gamma_2$. Consequently, $m =n$, and $\alpha_i$ and $\beta_i$ have the same target and source for $1 \leqslant i \leqslant n$. By the definition of biset product, the fact
\begin{equation*}
\alpha_n \circ (\alpha_{n-1} \ldots \circ \alpha_1) = \beta_n \circ (\beta_{n-1} \circ \ldots \circ \beta_1)
\end{equation*}
in the biset product implies that there is an automorphisms $g_{n-1} \in \text{Aut} _{\mathcal{C}} (x_{n-1})$ such that
\begin{equation*}
\alpha_n = \beta_n g_{n-1}, \quad (\alpha_{n-1} \circ \ldots \circ \alpha_1) = g_{n-1}^{-1} (\beta_{n-1} \circ \ldots \circ \beta_1),
\end{equation*}
where $x_{n-1}$ is the common target of $\alpha_{n-1}$ and $\beta_{n-1}$. By an easy induction on $n$, we show that $\{ \alpha_i \}_{i=1}^n$ and $\{ \beta_j \}_{j=1}^n$ have the required relations in the previous definition. Thus $\mathcal{C}$ satisfies the UFP.\

On the other hand, if $\mathcal{C}$ satisfies the UFP,  we want to show that $\mathcal{C}$ is isomorphic to the finite free EI category $\mathcal{C_{\hat{Q}}}$ generated from its finite EI quiver $\hat{Q}$. Define a functor $F: \mathcal{C} \rightarrow \mathcal{C_{\hat{Q}}}$ in the following way: First, $F(x) = x$ for every object $x$ in $\mathcal{C}$ since Ob$(\mathcal{C}) = \text{Ob} (\mathcal{C_{\hat{Q}}})$ by our construction. Furthermore, it is also clear that $\text{Aut}_{\mathcal{C}}(x) = \text{Aut} \mathcal{C_{\hat{Q}}} (x)$ for every object $x$, and the biset of unfactorizable morphisms from $x$ to $y$ in $\mathcal{C}$ is the same as that in $\mathcal{C_{\hat{Q}}}$ for every pair of different objects $x$ and $y$. Therefore we can let $F$ be the identity map restricted to automorphisms and unfactorizable morphisms in $\mathcal{C}$. Proposition 2.6 tells us that every non-isomorphism $\alpha$ in $\mathcal{C}$ is a composite $\alpha_n \circ \ldots \circ \alpha_1$ of unfactorizable morphisms, so $F(\alpha)$ can be defined as $F(\alpha_n) \circ \ldots \circ F(\alpha_1)$. By the UFP $F$ is well-defined and is a bijection restricted to Hom$_{\mathcal{C}} (x, y)$ for each pair of distinct objects $x$ and $y$. Consequently, $F$ is a bijection from Mor($\mathcal{C}$) to Mor$ (\mathcal{C_{\hat{Q}}})$ and so is an isomorphism. This finishes the proof.
\end{proof}

Finite free EI categories have a certain universal property which is stated in the following proposition:

\begin{proposition}
Let $\mathcal{C}$ be a finite EI category. Then there is a finite free EI category $\hat{\mathcal{C}}$ and a full functor $\hat{F}: \hat{\mathcal{C}} \rightarrow \mathcal{C}$ such that $\hat{F}$ is the identity map restricted to objects, isomorphisms and unfactorizable morphisms. This finite free EI category is unique up to isomorphism.
\end{proposition}

We call $\hat{\mathcal{C}}$ the \textit{free EI cover} of $\mathcal{C}$.\

\begin{proof} We mentioned before that every finite EI category $\mathcal{C}$ determines a finite EI quiver $\hat{Q}$ (see the paragraphs before Proposition 2.6), hence a finite free EI category $\mathcal{C_{\hat{Q}}}$ satisfying: Ob$(\mathcal{C}) = \text{Ob} (\mathcal{C_{\hat{Q}}})$; $\text{Aut}_{\mathcal{C}}(x) = \text{Aut} \mathcal{C_{\hat{Q}}} (x)$ for every object $x$; the biset of unfactorizable morphisms from $x$ to $y$ in $\mathcal{C}$ is the same as that in $\mathcal{C_{\hat{Q}}}$ for every pair of different objects $x$ and $y$.\

Define a functor $\hat{F}: \mathcal{C_{\hat{Q}}} \rightarrow \mathcal{C}$ in the following way: $F$ is the identity map on objects, isomorphisms and unfactorizable morphisms. Now if $\delta: x \rightarrow y$ is neither an isomorphism nor an unfactorizable morphism, it can be decomposed as the composite
\begin{equation*}
\xymatrix{x=x_0 \ar[r]^{\beta_1} & x_1 \ar[r]^{\beta_2} & \ldots \ar[r]^{\beta_m} &x_m=y,}
\end{equation*}
where each $\beta_i$ is unfactorizable for $1 \leqslant i \leqslant m$. Define
\begin{equation*}
\hat{F}(\delta) = \hat{F}(\beta_m) \ldots \hat{F}(\beta_2) \hat{F}(\beta_1).
\end{equation*}

$\hat{F}$ is clearly well defined for all isomorphisms and unfactorizable morphisms. We want to verify that $\hat{F}$ is well defined for factorizable morphisms as well. That is, if $\delta$ has another decomposition into unfactorizable morphisms
\begin{equation*}
\xymatrix{x=x_0 \ar[r]^{\beta'_1} & z_1 \ar[r]^{\beta_2} & \ldots \ar[r]^{\beta'_n} &z_n=y,}
\end{equation*}
then
\begin{equation*}
\hat{F}(\beta_n) \hat{F}(\beta_{n-1}) \ldots \hat{F}(\beta_1) = \hat{F}(\beta'_n) \ldots \hat{F}(\beta'_2) \hat{F}(\beta'_1).
\end{equation*}
Since $\mathcal{C}$ satisfies the UFP, we have $m=n$ and $x_i = z_i$ for $1 \leqslant i \leqslant n$, and $\beta_1 = h_1 \beta'_1$, $\beta_2 =h_2 \beta'_2 h_1^{-1}, \ldots, \beta_{n-1} = h_{n-1} \beta'_{n-1} h_{n-2}^{-1}$, $\beta_n =\beta'_n h_{n-1}^{-1}$, where $h_i \in \text{Aut}_{\mathcal{C}} (x_i)$. Thus:
\begin{align*}
& \hat{F}(\beta_n) \hat{F}(\beta_{n-1}) \ldots \hat{F}(\beta_1) \\
& = \hat{F}(\beta'_n h_{n-1}^{-1}) \hat{F}(h_{n-1} \beta'_{n-1} h_{n-2}^{-1}) \ldots \hat{F}(h_1 \beta'_1) \\
& = \hat{F}(\beta'_n) \hat{F}(h^{-1}_{n-1}) \hat{F}(h_{n-1}) \ldots \hat{F}(h_1^{-1}) \hat{F}(h_1) \hat{F}(\beta'_1) \\
& = \hat{F}(\beta'_n) \ldots \hat{F}(\beta'_2) \hat{F}(\beta'_1).
\end{align*}

Therefore, $\hat{F}$ is a well defined functor. It is full since all automorphisms and unfactorizable morphisms in $\mathcal{C}$ are images of $\hat{F}$, and all other morphisms in $\mathcal{C}$ are their composites. Moreover, $\hat{\mathcal{C}}$ is unique up to isomorphism since it is completely determined by objects, isomorphisms and unfactorizable morphisms in $\mathcal{C}$.\
\end{proof}

It is well-known that every subgroup of a free group is still free. Finite free EI categories have a similar property.

\begin{proposition}
Let $\mathcal{C}$ be a finite EI category. Then $\mathcal{C}$ is a finite free EI category if and only if all of its full subcategories are finite free EI categories.
\end{proposition}

\begin{proof}
The if part is trivial since $\mathcal{C}$ is such a subcategory of itself. Now let $\mathcal{D}$ be a full subcategory of $\mathcal{C}$. We want to show that $\mathcal{D}$ satisfies the UFP.\

Take a factorizable morphism $\alpha$ in $\mathcal{D}$ and two decompositions of the following form:
\begin{equation*}
\xymatrix{x=x_0 \ar[r]^{\alpha_1} & x_1 \ar[r]^{\alpha_2} & \ldots \ar[r]^{\alpha_m} &x_m=y \\
x=x_0 \ar[r]^{\beta_1} & x'_1 \ar[r]^{\beta_2} & \ldots \ar[r]^{\beta_n} & x'_n=y}
\end{equation*}
where all $\alpha_i$ and $\beta_j$ are unfactorizable in $\mathcal{D}$, but possibly factorizable in $\mathcal{C}$. Decomposing them into unfactorizable morphisms in $\mathcal{C}$, we get two extended sequences of unfactorizable morphisms as follows:

\begin{equation*}
\xymatrix{x=x_0 \ar[r]^{\delta_1} & w_1 \ar[r]^{\delta_2} & \ldots \ar[r]^{\delta_r} &w_r=y \\
x=x_0 \ar[r]^{\theta_1} & w_1 \ar[r]^{\theta_2} & \ldots \ar[r]^{\theta_r} & w_r=y}
\end{equation*}
Each pair of morphisms $\delta_i$ and $\theta_i$ have the same source and target since $\mathcal{C}$ is a finite free EI category. Moreover, there are $h_i \in \text{Aut} _{\mathcal{C}} (w_i)$, $1 \leqslant i \leqslant r-1$, such that $\theta_1 = h_1 \delta_1$, $\theta_2 = h_2 \delta_2 h_1^{-1}$,  $\ldots$, $\theta_{r-1} = h_{r-1} \delta_{r-1} h_{r-2}^{-1}$, $\theta_r = h_{r-1}^{-1} \delta_r$. If we can prove the fact that $m=n$ and $x_1 = x_1'$, $\ldots$, $x_m = x_n'$, then the conclusion follows. Indeed, if this is true, say $x_1 = x_1' = w_{r_1}$, $\ldots$, $x_m = x'_m = w_{r_m}$, then we have $\beta_1 = h_{r_1} \alpha_1$, $\beta_2 = h_{r_2} \alpha_2 h_{r_1}^{-1}$, $\ldots$, $\beta_n = \alpha_n h_{r_n}^{-1}$, which is exactly the UFP.\

We show this fact by contradiction. Suppose that $x_1 = x'_1$, $x_2 = x_2'$, $\ldots$, $x_{i-1} = x'_{i-1}$, and let $x_i$ be the first object in the sequence different from $x_i'$. Notice both $x_m \ldots x_1$ and $x_n' \ldots x_1'$ are subsequences of the sequence $w_r \ldots w_1$. As a result, $x_i$ appears before $x_i'$ in $w_r \ldots w_1$, or after $x_i'$. Without loss of generality we assume that $x_i$ is before $x_i'$. Let $x_{i-1} = x'_{i-1} = w_a$, $x_i= w_b$ and $x_i'= w_c$. We must have $a \leq b < c$. Furthermore, we check that
\begin{equation*}
\beta_i = \theta_c \circ \theta_{c-1} \circ \ldots \circ \theta_{a+1} = (\theta_c \circ \ldots \circ \theta_{b+1}) \circ (\theta_b \circ \ldots \circ \theta_{a+1}),
\end{equation*}
with $(\theta_c \circ \ldots \circ \theta_{b+1})$ contained in Hom$_{\mathcal{C}} (x_i, x_i')$ and $(\theta_b \circ \ldots \circ \theta_{a+1})$ contained in Hom$_{\mathcal{C}} (x_{i-1}, x_i)$. But $\mathcal{D}$ is a full subcategory, so $\beta_i$ is the composite of two non-isomorphisms in $\mathcal{D}$. This is a contradiction since we have assumed that $\beta_i$ is unfactorizable in $\mathcal{D}$.
\end{proof}

The condition that $\mathcal{D}$ is a full subcategory of $\mathcal{C}$ is required. Consider the following two examples:

\begin{example}
Let $\mathcal{C}$ be the free category generated by the following quiver. Let $\mathcal{D}$ be the subcategory of $\mathcal{C}$ obtained by removing the morphism $\beta$ from $\mathcal{C}$. The category $\mathcal{C}$ is free, but the subcategory $\mathcal{D}$ is not free. Indeed, morphisms $\alpha$, $\beta\alpha$, $\gamma \beta$, $\gamma$ are all unfactorizable in $\mathcal{D}$. Thus the morphism $\gamma \beta \alpha$ has two decompositions $(\gamma\beta) \circ \alpha$ and $(\gamma)\circ (\beta\alpha)$, which contradicts the UFP.
\end{example}

\begin{equation*}
\xymatrix{ \bullet \ar[r]^{\alpha} & \bullet \ar[r]^{\beta} & \bullet \ar[r]^{\gamma} & \bullet }
\end{equation*}

\begin{example}
Let $\mathcal{C}$ be the following category, where $g$ generates a cyclic group of order 2, interchanges $\beta_1$ and $\beta_2$, and fixes $\alpha$. Then $\beta_1 \alpha = (\beta_2 g) \alpha = \beta_2 (g \alpha) = \beta_2 \alpha$. It is not hard to check that $\mathcal{C}$ satisfies the UFP; but the subcategory formed by removing the morphism $g$ from $\mathcal{C}$ doesn't satisfies the UFP.
\end{example}

\begin{equation*}
\xymatrix{ 1_x \ar[r]^{\alpha} & \langle g \rangle \ar@<0.5ex>[r]^{\beta_1} \ar@<-0.5ex>[r]_{\beta_2} & 1_z }
\end{equation*}

\begin{proposition}
Let $\mathcal{C}$ be a finite EI category and $\hat{\mathcal{C}}$ be its free EI cover. Then the category algebra $k\mathcal{C}$ is a quotient algebra of $k \hat {\mathcal{C}}$. In particular, if $\hat {\mathcal{C}}$ is of finite representation type, so is $\mathcal{C}$.
\end{proposition}

\begin{proof}
Since the functor $\hat{F}: \hat{\mathcal{C}} \rightarrow \mathcal{C}$ is bijective on objects, by Proposition 2.2.3 in ~\cite{Xu2}, it induces an algebra homomorphism from $k\hat{\mathcal{C}}$ to $k\mathcal{C}$ by sending $\alpha \in \text{Mor} (\hat{\mathcal{C}})$ to $\hat{F} (\alpha) \in \text{Mor} (\mathcal{C})$. This homomorphisms is surjective since $\hat{F}$ is full.
\end{proof}

We want to clarify some confusion probably caused by the name ``finite free EI categories": a finite free EI category $\mathcal{C}$ is in general not a free category as the following example shows.\

\begin{example} Let $\mathcal{C}$ be a category with one object whose morphisms form a non-identity group. Then $\mathcal{C}$ is a finite free EI category by our definition. But it is not a free category since the morphism group is not a free monoid.
\end{example}

\section{Representations of Free EI Categories}

In this section we study the representations of finite free EI categories. First, let us introduce some notation.\

A representative unfactorizable morphism $\alpha: x\rightarrow y$ uniquely determines a subcategory $\mathcal{D}_{\alpha}$ of $\mathcal{C}$: Ob$(\mathcal{D}_{\alpha}) = \{x,y\}$, Aut$_{\mathcal{D}_{\alpha}}(x) = \text{Aut}_{\mathcal{C}}(x)$, Aut$_{\mathcal{D}_{\alpha}}(y) = \text{Aut}_{\mathcal{C}}(y)$, Hom$_{\mathcal{D}_{\alpha}}(x,y)$ is the whole (Aut$_{\mathcal{C}} (y)$, Aut$_ {\mathcal{C}} (x)$)-orbit where $\alpha$ lies. Obviously, $\mathcal{D}_{\alpha}$ is a finite free EI category.\

With the above setup, we can show that a representation of a finite free EI category $\mathcal{C}$ is determined by its local structures on subcategories $\mathcal{D}_{\alpha}$ for all representative unfactorizable morphisms $\alpha$. Stated formally, it is:

\begin{proposition}
Let $\mathcal{C}$ be a finite free EI category, and $R$ be a rule assigning a $k \text{Aut}_{\mathcal{C}} (x)$-module to each object $x$ and a linear transformation to each unfactorizable morphism in $\mathcal{C}$. If $R$ restricted to $\mathcal{D}_{\alpha}$ is a representation of $\mathcal{D}_{\alpha}$ for every representative unfactorizable morphism $\alpha$ in $\mathcal{C}$, then $R$ induces a unique representation $\tilde{R}$ of $\mathcal{C}$ such that $\tilde{R}$ restricted to each $\mathcal{D}_{\alpha}$ is the same as $R$ restricted to $\mathcal{D}_{\alpha}$.
\end{proposition}

\begin{proof}
Since we can define $\tilde{R}(x)=R(x)$ for each object $x$ in $\mathcal{C}$, it suffices to define a linear map for each morphism in $\mathcal{C}$ such that the composition relations are preserved. If $\beta$ is an automorphism or an unfactorizable morphism, it must belong to some subcategory $\mathcal{D}_{\alpha}$ for a representative unfactorizable morphism $\alpha$. Because $R$ restricted to $\mathcal{D}_{\alpha}$ is a representation, we then define $\tilde{R} (\beta) = R(\beta)$.\

Now suppose that $\beta$ is factorizable and decompose it into unfactorizable morphisms of the following form
\begin{equation*}
\xymatrix{x=x_0 \ar[r]^{\alpha_1} & x_1 \ar[r]^{\alpha_2} & \ldots \ar[r]^{\alpha_m} &x_m=y}.
\end{equation*}
Let $\tilde{R}(\alpha) = R(\alpha_m) \ldots R(\alpha_1)$. We claim that it is well defined. That is: if $\alpha$ has two different decompositions
\begin{equation*}
\xymatrix{x=x_0 \ar[r]^{\alpha_1} & x_1 \ar[r]^{\alpha_2} & \ldots \ar[r]^{\alpha_m} &x_m=y}
\end{equation*}
and
\begin{equation*}
\xymatrix{x=x_0 \ar[r]^{\beta_1} & y_1 \ar[r]^{\beta_2} & \ldots \ar[r]^{\beta_n} &y_n=y},
\end{equation*}
then it is true that $R(\alpha_m) \ldots R(\alpha_1) = R(\beta_n) \ldots R(\beta_1)$.\

Indeed, because $\mathcal{C}$ satisfies the unique factorization property, we have $m=n$, $x_i = y_i$, and $\beta_1 = h_1 \alpha_1$, $\beta_2 = h_2 \alpha_2 h_1^{-1}$, \ldots, $\beta_{n-1} = h_{n-1} \alpha_{n-1} h_{n-2}^{-1}$, $\beta_n = \alpha_n h_{n-1}^{-1}$ for some $h_i \in \text{End} _{\mathcal{C}} (x_i)$, $1 \leqslant i \leqslant n-1$. Notice that each $\alpha_i$ and the corresponding $\beta_i$ are unfactorizable, and lie in $\mathcal{D}_{\alpha_i}$ since $\beta_i$ and $\alpha_i$ are in the same orbit. The fact that $R$ restricted to $\mathcal{D}_{\alpha_i}$ is a representation of this subcategory implies $R(h_i) R(h_i^{-1}) = R(h_i h_i^{-1}) = 1$. Thus:
\begin{align*}
& \tilde{R}(\beta_n) \ldots R(\beta_2) R(\beta_1)\\
& = R(\alpha_n) R(h_{n-1}^{-1}) R(h_{n-1}) R(\alpha_{n-1}) R(h_{n-2}^{-1})  \ldots  R(h_1) R(\alpha_1) \\
& = R(\alpha_n) \ldots R(\alpha_1) \\
& = \tilde{R}(\alpha_n \ldots \alpha_1).
\end{align*}

We proved that $\tilde{R}$ is indeed a representation of $\mathcal{C}$. Moreover, $\tilde{R}$ restricted to any $\mathcal{D}_{\alpha}$ is exactly the same as $R$ restricted to this subcategory. The uniqueness is obvious since from the above process we conclude that $\tilde{R}$ is uniquely determined by its values on subcategories $\mathcal{D}_{\alpha}$, for all representative unfactorizable morphisms $\alpha$ in $\mathcal{C}$.
\end{proof}

\begin{proposition}
Let $\mathcal{C}$ be a finite free EI category. Let $R_1$ and $R_2$ be two representations of $\mathcal{C}$. Then a family of $ k\text{Aut}_{\mathcal{C}}(x)$-module homomorphisms $\{ \phi_x: R_1(x) \rightarrow R_2(x) \mid x \in \text{Ob}(\mathcal{C}) \}$ is a $k\mathcal{C}$-module homomorphism from $R_1$ to $R_2$ if and only if for each representative unfactorizable morphism $\alpha: x \rightarrow y$, we have $\phi_y R_1(\alpha) = R_2(\alpha) \phi_x$.
\end{proposition}

\begin{proof}
The only if part is trivial. For the other direction, let us consider unfactorizable morphisms first. If $\beta: x \rightarrow y$ is unfactorizable, it lies in the same orbit as a unique representative unfactorizable morphism $\alpha$, i.e., $\beta =h\alpha g$, where $g \in \text{Aut}_{\mathcal{C}}(x)$ and $h \in \text{Aut}_{\mathcal{C}}(y)$. Because $\phi_x$ is a $k \text{Aut}_{\mathcal{C}}(x)$-module homomorphism, and $\phi_y$ is a $k \text{Aut}_{\mathcal{C}}(y)$-module homomorphism, we have:
\begin{align*}
\phi_y R_1(h\alpha g) & = \phi_y R_1(h) R_1(\alpha) R_1(g) \\
& = R_2(h) \phi_y R_1(\alpha) R_1(g)\\
& = R_2(h) R_2(\alpha) \phi_x R_1(g)\\
& = R_2(h) R_2(\alpha) R_2(g) \phi_x\\
& = R_2(h \alpha g) \phi_x.
\end{align*}
That is, $\phi_y R_1(\beta) = R_2(\beta) \phi_x$.\

If $\beta$ is factorizable, let \xymatrix{x=x_0 \ar[r]^{\alpha_1} & x_1 \ar[r]^{\alpha_2} & \ldots \ar[r]^{\alpha_m} &x_m=y} be a decomposition of $\beta$ into unfactorizable morphisms. Consider the following diagram:\
\begin{align*}
\centerline{
\xymatrix{R_1(x_0) \ar[r]^{R_1(\alpha_1)} \ar[d]^{\phi_0} & R_1(x_1) \ar[r]^{R_1(\alpha_2)} \ar[d]^{\phi_1} & \ldots \ar[r]^{R_1(\alpha_m)} &R_1(x_m) \ar[d]^{\phi_m} \\
R_2(x_0) \ar[r]^{R_2(\alpha_1)} & R_2(x_1) \ar[r]^{R_2(\alpha_2)} & \ldots \ar[r]^{R_2(\alpha_m)} &R_2(x_m).}
}
\end{align*}
For $i=0,\ldots,m-1$, we have $\phi_{i+1} R_1(\alpha_{i+1}) = R_2(\alpha_{i+1}) \phi_i$. Thus:
\begin{align*}
& \phi_m R_1(\beta)\\
& = \phi_m R_1(\alpha_m) R_1(\alpha_{m-1}) \ldots  R_1(\alpha_1) \\
& = R_2(\alpha_m) \phi_{m-1} R_1(\alpha_{m-1}) \ldots R_1(\alpha_1) \\
& \ldots \\
& =R_2(\alpha_m) R_2(\alpha_{m-1}) \ldots R_2(\alpha_1) \phi_0 \\
& =R_2(\beta) \phi_0.
\end{align*}

This finishes the proof.
\end{proof}

Throughout the remaining of this section we use the following notation: $\alpha: x \rightarrow y$ is a fixed representative unfactorizable morphism and $\mathcal{D}_{\alpha}$ is the corresponding subcategory of $\mathcal{C}$ (see our notation before Proposition 3.1 for the definition of $\mathcal{D}_{\alpha}$). Let $G$ and $H$ be Aut$_{\mathcal{C}}(x)$ and Aut$_{\mathcal{C}}(y)$ respectively. Define: $G_0=\text{Stab}_{G}(\alpha)$, $H_0=\text{Stab}_{H}(\alpha)$, $G_1=\text{Stab}_{G}(H\alpha)$, and $H_1=\text{Stab}_H(\alpha G)$. Obviously $G_1=\{g \in G: \text{exists } h\in H \text{ with }\alpha g=h\alpha\}$, and $H_1= \{h \in H: \text{exists } g\in G \text{ with } h\alpha = \alpha g\}$. The group $G$ ($H$, resp.) acts transitively on Hom$_{\mathcal{D}_{\alpha}} (x,y)$ if and only if $H_1=H$ ($G_1=G$, resp.), bearing in mind that Hom$_{\mathcal{D}_{\alpha}} (x,y)$ only has one orbit as a biset.

\begin{lemma}
With the above notation, $G_0$ and $G_1$ are subgroups of $G$, $H_0$ and $H_1$ are subgroups of $H$. Moreover, $G_0\lhd G_1$, $H_0\lhd H_1$, and $G_1/G_0 \cong H_1/H_0$.
\end{lemma}

\begin{proof}
This result is well-known to Bouc who described the structures of bisets in ~\cite{Bouc}. We give an elementary proof here. It is routine to check that $G_0 \leqslant G$, $G_1 \leqslant G$ by definition. Similarly, $H_0 \leqslant H$ and $H_1 \leqslant H$. It is also clear that $G_0 \leqslant G_1$ and $H_0 \leqslant H_1$.\

Let $g_0 \in G_0$, $g\in G_1$ and consider $g^{-1}g_0g$. Since $g\in G_1$, there is $h\in H$ such that $\alpha g=h\alpha$. Then $\alpha g^{-1}g_0g=h^{-1}\alpha g_0 g=h^{-1}\alpha g=h^{-1}h\alpha=\alpha$. This means $g^{-1}g_0g\in G_0$ and $G_0\lhd G_1$. Similarly, $H_0 \lhd H_1$.\

Now define a map $\phi: G_1 \mapsto H_1/H_0$ by the following rule: for every $g\in G_1$, there is some $h\in H$ such that $\alpha g=h\alpha$. By the definition $h$ is contained in $H_1$. Define $\phi(g)=\bar{h}$, the image of $h$ in $H_1/H_0$. The reader can check that $\phi$ is well defined, surjective, and a group homomorphism. Moreover, the kernel of this map is exactly $G_0$. Thus $G_1/G_0 \cong H_1/H_0$.
\end{proof}

\begin{remark}
Under the isomorphism given in this lemma, we identify the quotient group $G_1/G_0$ with $H_1/H_0$, and the module $k\uparrow_{G_0}^{G_1} \cong k(G_1/G_0)$ with $k\uparrow_{H_0}^{H_1} \cong k(H_1/H_0)$.
\end{remark}

From now on we insist:\\

\noindent \textbf{Convention:} \textit{All finite EI categories $\mathcal{C}$ we study satisfy that the endomorphism group of every object $x$ in $\mathcal{C}$ has order invertible in $k$.}\\

A representation $R$ of $\mathcal{D}_{\alpha}$ determines a linear map $\varphi= R(\alpha)$ from a $kG$-module $R(x)$ to a $kH$-module $R(y)$. Notice that both $R(x)$ and $R(y)$ are semisimple by Maschke's theorem in view of the previous convention.

\begin{lemma}
Let $U$ be a simple summand of $R(x)\downarrow_{G_1}^G$. If $\varphi(U) \neq 0$, then:\\
(1) $U$ is a direct summand of $k\uparrow_{G_0}^{G_1}$;\\
(2) the restricted map $\varphi_U: U \rightarrow \varphi(U)$ is a $k(G_1/G_0)$-module isomorphism under the identification $k\uparrow_{G_0}^{G_1} \cong k\uparrow _{H_0}^{H_1}$;\\
(3) $\varphi(U)$ is a direct summand of $k\uparrow_{H_0}^{H_1}$.
\end{lemma}

\begin{proof}
If $\varphi(U) \neq 0$, there is some $u \in V_1 \subseteq U$ such that $\varphi(u) = w \neq 0$, where $V_1$ is a simple summand of $U\downarrow_{G_0}^{G_1}$. We claim that each $g \in G_0$ fixes $u$, hence $V_1 \cong k$. Indeed, if there is $g\in G_0$ such that $gu\neq u$, then $\varphi(u-gu)=\varphi(u) - (\varphi g)(u)=0$ since $g\in G_0$ fixes $\varphi$. But $u-gu \neq 0$ is also in $V_1$ and generates the simple $kG_0$-module $V_1$, too. Now the fact that $G_0$ fixes $\varphi$ and $\varphi(u-gu)=0$ implies that $V_1$ is sent to 0 by $\varphi$. In particular $u$ is sent to 0. This is a contradiction. Thus $V_1 \cong k \mid U\downarrow_{G_0}^{G_1}$. But Hom$_{kG_1}(k\uparrow_{G_0}^{G_1}, U) \cong \text{Hom}_{kG_0} (k, U\downarrow_{G_0} ^G) \neq 0$, so $U \mid k\uparrow_{G_0}^{G_1}$. This proves the first statement.\

Take any $g\in G_1$ and let $\bar{g}$ be its image in $G_1/G_0$ under the natural projection. Since $U \mid k\uparrow _{G_0}^{G_1}$, we know $gu = \bar{g}u$ for every $u \in U$. With the identification $G_1/G_0 \cong H_1/H_0$, $\bar{g}$ is identified with a certain $\bar{h} \in H_1/H_0$, where $h \in H_1$ acts on $\varphi(U)$ in the same way as $\bar{h}$ acts on it. But $h\alpha = \alpha g \in \text{Hom}_{\mathcal{D}_{\alpha}} (x, y)$ and $R$ is a representation of $\mathcal{D_{\alpha}}$, so we should have $\varphi(gu) = h \varphi(u)$, that is: $\varphi(\bar{g}u) = \bar{h}\varphi(u)$. Consequently, $\varphi$ is a $k(G_1/G_0)$-module homomorphism if we identify $\bar{g}$ and $\bar{h}$. Since $U$ is simple, and $\varphi$ is nonzero, it must be an isomorphism. This is the second statement.

On the other hand, every $h \in H_0$ acts trivially on $\varphi(u) = w$ since $h$ fixes $\varphi$. Thus the $kH_0$-module generated by $w$ is isomorphic to the trivial $kH_0$-module $k$, so $k \mid \varphi(U) \downarrow_{H_0}^{H_1}$. Again, Hom$_{kH_1}(k\uparrow_{H_0}^{H_1}, \varphi(U)) \cong \text{Hom}_{kH_0} (k, \varphi(U) \downarrow_{H_0} ^{H_1}) \neq 0$, so $\varphi(U)$ and $k\uparrow_{H_0}^{H_1}$ have a common simple summand. By (2), $\varphi(U)$ as the nontrivial homomorphic image of a simple module must be simple, too. Thus $\varphi(U) \mid k \uparrow_{H_0}^{H_1}$.\
\end{proof}

The conclusions in the previous lemma actually characterize linear maps $R(\alpha)$, where $R$ is a representation of $\mathcal{D}_{\alpha}$.

\begin{lemma}
let $\varphi: M \rightarrow N$ be a linear map from a $kG$-module $M$ to a $kH$-module $N$ such that for every simple summand $U$ of $M \downarrow _{G_1}^G$, the restricted map $\varphi_U: U \rightarrow \varphi(U)$ is zero when $U \nmid k\uparrow_{G_0}^{G_1}$ , and a $k(G_1/G_0)$-module homomorphism otherwise. Then $\varphi$ determines a representation $R$ of $\mathcal{D}_{\alpha}$ with $R(x)= M$, $R(y) = N$ and $R(\alpha) = \varphi$.
\end{lemma}

\begin{proof}
We define a functor $R: \mathcal{D}_{\alpha} \rightarrow k$-Vec by letting $R(x) = M$, $R(y) = N$. If $\beta$ is an automorphism of $x$ or $y$, $R(\beta)$ is already defined since $M$ is a $k$Aut$_{\mathcal{D}_{\alpha}} (x)$-module and $N$ is a $k \text{Aut} _{\mathcal{D}_{\alpha}} (y)$-module. Otherwise $\beta$ in Hom$_{\mathcal{D}_{\alpha}} (x,y)$ can be written as $h\alpha g$, where $h \in H$, $g \in G$, $\alpha$ is the representative unfactorizable morphism determining $\mathcal{D}_{\alpha}$. In this case $R(\beta)$ can be defined as $h \varphi g$. The conclusion follows after we show that this definition is well-defined. That is, if $h_1\alpha g_1 = h_2 \alpha g_2$, then $h_1 \varphi g_1 = h_2 \varphi g_2$.\

Since $M$ as a vector space is the direct sum of all the summands of $M \downarrow_{G_1}^G$, it is enough to show that for every summand $U$, and each $u \in U$, we have $h_1\varphi g_1(u) = h_2\varphi g_2 (u)$, or equivalently: $h_2^{-1} h_1 \varphi(u) = \varphi (g_2 g_1^{-1}u)$. The fact that $h_1\alpha g_1= h_2 \alpha g_2$ implies $h_2^{-1}h_1 \alpha =\alpha g_2 g_1^{-1}$, hence $h_2^{-1}h_1 \in H_1$ and $g_2g_1^{-1} \in G_1$, so $g_2g_1^{-1}u \in U$ and $\varphi(g_2g_1^{-1}u) = \varphi_U(g_2g_1^{-1}u)$. Thus it is sufficient to prove that $h_2^{-1} h_1 \varphi_U(u) = \varphi_U (g_2 g_1^{-1}u)$.\

If $\varphi_U=0$, we have $\varphi_U(g_2g_1^{-1}u)= 0 = h_2^{-1}h_1 0 = h_2^{-1}h_1\varphi_U(u)$ trivially. Otherwise, let $\bar{h}$ be the projection of $h_2^{-1}h_1$ into $H_1/H_0$, and $\bar{g}$ be the projection of $g_2^{-1}g_1$ into $G_1/G_0$. Then $g_2^{-1}g_1u = \bar{g}u$, and $h_2^{-1}h_1\varphi_U(u) = \bar{h} \varphi_U(u)$ because $U \mid k\uparrow_{G_0}^{G_1}$ and $\varphi(U) \mid k\uparrow_{H_0}^{H_1}$. Since $\bar{g}$ is sent to $\bar{h}$ by the isomorphism $G_1/G_0 \cong H_1/H_0$, and $\varphi_U$ is a $k(G_1/G_0)$-module homomorphism, we have: $h_2^{-1}h_1\varphi_U(u) = \bar{h} \varphi_U(u) = \varphi_U(\bar{g}u) = \varphi_U(g_2^{-1}g_1u)$. Thus $R$ is well defined, and the conclusion follows.
\end{proof}

\begin{remark}
Let $\varphi: M  \rightarrow N$ be a linear map from a $kG$-module $M$ to a $kH$-module $N$. Let $U$ ($T$, resp.) be a simple summand of $M\downarrow_{G_1}^G$ ($N\downarrow_{H_1}^H$, resp.). Take a fixed decomposition $N = T \oplus N'$ of vector space and let $p: N \rightarrow T$ be the projection. Then $\varphi$ is a direct sum of $\varphi_{UT}$s of the following form:
\begin{equation*}
\xymatrix{ \varphi_{UT}: U \ar@{^{(}->}[r] & M \ar[r]^{\varphi} & N \ar@{->>}[r]^p & T}
\end{equation*}
Lemma 3.5 and Lemma 3.6 tell us that $\varphi$ is given by or gives a representation $R$ of $\mathcal{D}_{\alpha}$ if and only if the following is true: whenever $\varphi_{UT} \neq 0$, then $U \mid k\uparrow_{G_0}^{G_1}$, $T \mid k\uparrow_{H_0}^{H_1}$, and $\varphi_{UT}$ is a $k(G_1/G_0)$-module isomorphism under the given identification.
\end{remark}

\section{Ordinary Quivers of Finite EI Categories}

In this section we will construct the ordinary quiver $Q$ for every finite EI category $\mathcal{C}$ under the hypothesis that the endomorphism groups of all objects in $\mathcal{C}$ have orders invertible in $k$. The algorithm used here to construct $Q$ is expressed in terms of the simple modules of endomorphism groups of objects in $\mathcal{C}$ and their restrictions to subgroups. This is easier and more intuitive than the usual approach which uses the primitive idempotents and the radical of $k\mathcal{C}$. Let us first introduce some preliminary results.

Let $G$ be a finite group whose order is invertible in $k$ and $G_1 \leqslant G$ be a subgroup. For every $kG$-module $M$, we denote the homogeneous space of a simple $kG$-module $V$ in $M$ by $M(V)$, that is, the sum of all submodules of $M$ isomorphic to $V$. If $U$ is a simple summand of $V\downarrow _{G_1}^{G}$, we denote the homogeneous space of $U$ in $M(V)\downarrow_{G_1}^{G}$ by $M(V,U)$. Since $V \downarrow_{G_1}^G$ may contain more than one simple summand isomorphic to $U$, we choose a particular decomposition and index different isomorphic copies of $U$ in $V\downarrow_{G_1}^G$ by natural numbers $s \in \mathbb{N}$. That is:
\begin{equation*}
V \downarrow _{G_1}^G = \bigoplus _{s=1}^t U_s \oplus V',
\end{equation*}
where each $U_s$ is isomorphic to $U$ and $V'$ has no summands isomorphic to $U$.\

Now let $M(V,U,s)$ be the sum of the $s$-th isomorphic copy of $U$ in each $V\downarrow_{G_1}^G$ of $M(V)\downarrow_{G_1}^{G}$, $s \geqslant 1$. Thus if
\begin{equation*}
M(V) = V_1 \oplus V_2 \oplus \ldots \oplus V_m
\end{equation*}
is any decomposition of $M(V)$ and $\theta_i: V \rightarrow V_i$ is a $kG$-module isomorphism, and
\begin{equation*}
V\downarrow_{G_1}^G = U_1 \oplus \ldots \oplus U_t \oplus V'
\end{equation*}
is some fixed decomposition where $U_s \cong U$ for all $1 \leqslant s \leqslant t$ and $V'$ has no summands isomorphic to $U$ as $kG_1$-modules, then $M(V, U, s)$ is defined as $\theta_1 (U_s) \oplus \theta_2(U_s) \oplus \ldots \oplus \theta_m(U_s)$. This definition of $M(V, U, s)$ is dependent on the fixed decomposition of $V \downarrow_{G_1}^G$, but independent of the decomposition of $M(V)$ up to an automorphism of $M(V)$, or equivalently, independent of the particular isomorphisms $\theta_i$. Indeed, if
\begin{equation*}
M(V) = V_1 \oplus V_2 \oplus \ldots \oplus V_m
\end{equation*}
and
\begin{equation*}
M(V) = W_1 \oplus W_2 \ldots \oplus W_n
\end{equation*}
are two decompositions of $M(V)$. Then we get two lists of isomorphisms $\{ \theta_i \}_{i=1}^m$ and $\{ \psi_j\} _{j=1}^n$ with $\theta_i(V) = V_i$ and $\psi_j(V) = W_j$. By the Krull-Schmidt Theorem, we have $m = n$, and there is a permutation $\pi$ of $\{1, \ldots, n\}$ together with a family of $kG$-module isomorphisms $\phi_i: V_i \cong W_{\pi(i)}$. Since both $\psi_i$ and $\phi_i \theta_i$ are isomorphisms from the simple $kG$-module $V$ to the simple $kG$-module $W_i$, we know that $\psi_i$ is a scalar multiple of $\phi_i \theta_i$ by Schur's lemma. Thus
\begin{align*}
& \psi_1(U_s) \oplus \psi_2(U_s) \ldots \oplus \psi_m(U_s) \\
& = \phi_1 (\theta_1 (U_s)) \oplus \phi_2(\theta_2 (U_s)) \oplus \ldots \oplus \phi_m (\theta_m (U_s)) \\
& = \phi ( \theta_1(U_s) \oplus \theta_2(U_s) \oplus \ldots \oplus \theta_m(U_s) )
\end{align*}
where $\phi$ is the direct sum of $\phi_i$, an automorphism of $M(V)$. Thus the automorphism $\phi$ sends the space $M(V,U,s)$ determined by the first decomposition of $M(V)$ to the space $M(V,U,s)$ determined by the second decomposition of $M(V)$.\

Repeating the above process for all simple summands $V$ of $M$, and all simple summands $U$ of $V\downarrow_{G_1}^G$, we can decompose $M$ as the direct sum of all subspaces of the form $M(V,U,s)$. That is:
\begin{equation*}
M = \bigoplus _{V|M} \bigoplus_ {U \mid V \downarrow_{G_1}^G}  \bigoplus _{s \in \mathbb{Z}} M(V,U,s).
\end{equation*}
Notice that for each fixed $V$ and $U$, there are only finitely many $s$ such that $M(V,U,s) \neq 0$. Thus this decomposition is well-defined. Let $p_{V,U,s}: M \rightarrow M(V,U,s)$ be the projection.\

\begin{lemma}
With the above setup, every $kG$-module homomorphism $\phi: M_1 \rightarrow M_2$ as a linear map is the direct sum of all $\phi_{V,U,s}$.
\begin{equation*}
\xymatrix{\phi_{V,U,s}: M_1(V,U,s)\ar@{^{(}->}[r] & M_1 \ar[r]^{\phi} & M_2 \ar@{->>}[r] ^-{p_{V,U,s}} & M_2(V,U,s)}
\end{equation*}
\end{lemma}

\begin{proof}
Since $M_1$ as a vector space is the direct sum of all subspaces of the form $M_1(V,U,s)$, it suffices to show that every subspace $M_1(V,U,s)$ is sent by $\phi$ into $M_2(V,U,s)$. First, since $\phi$ is a module homomorphism, and both $M_1$ and $M_2$ are semisimple, then the homogeneous space $M_1(V) \subseteq M_1$ is sent into the homogeneous space $M_2(V) \subseteq M_2$. If either $M_1(V)$ or $M_2(V)$ is 0, i.e., either $M_1$ or $M_2$ has no summand isomorphic to $V$, the conclusion holds trivially. Otherwise, take a particular simple summand $V_1$ of $M_1(V)$, and a particular simple summand $V_2$ of $M_2(V)$. The induced $kG$-module homomorphism from $V_1$ to $V_2$ by $\phi$ is a scalar multiplication, hence sends the $s$-th isomorphic copy of $U$ in $V_1 \downarrow_{G_1}^G$ into the $s$-th isomorphic copy of $U$ in $V_2\downarrow_{G_1}^G$. Consequently, $M_1(V,U,s)$ is sent into $M_2(V,U,s)$.
\end{proof}

Now we construct the ordinary quiver $Q$ for a finite EI category $\mathcal{C}$. The detailed algorithm is as follows:\\

\noindent \textbf{Step 1}: The vertex set of $Q$ is $\bigsqcup_{x \in \text{Ob} (\mathcal{C})} S_x$, where $S_x$ is a set of representatives of the isomorphism classes of simple $k \text{Aut} _{\mathcal{C}} (x)$-modules.\\

\noindent \textbf{Step 2}: Let $\alpha: x \rightarrow y$ be a representative unfactorizable morphism. Then it determines uniquely:\\
$\bullet$ $G = \text{Aut}_{\mathcal{C}}(x)$, $G_0 = \text{Stab}_{G}(\alpha)$, $G_1 = \text{Stab}_{G}(H\alpha)$;\\
$\bullet$ $H = \text{Aut}_{\mathcal{C}}(y)$, $H_0 = \text{Stab}_{H}(\alpha)$, $H_1 = \text{Stab}_{H}(\alpha G)$;\\
$\bullet$ $\{V_1, \ldots, V_m\}$: the set of pairwise non-isomorphic simple $kG$-modules;\\
$\bullet$ $\{W_1, \ldots, W_n\}$: the set of pairwise non-isomorphic simple $kH$-modules;\\
$\bullet$ $\{U_1, \ldots, U_r\}$: the set of pairwise non-isomorphic simple summands of $k\uparrow_{G_0}^{G_1}$.\\

\noindent \textbf{Step 3}: For each particular simple $kG$-module $V$ in $\{V_1, \ldots, V_m\}$ choose a decomposition $V\downarrow_{G_1}^{G} \cong U_1^{e_1} \oplus \ldots \oplus U_r^{e_r} \oplus X$ where $X$ has no summand isomorphic to any $U_i$. For each simple $kH$-module $W$ in $\{W_1, \ldots, W_n\}$ choose a decomposition $W\downarrow_{H_1}^{H} \cong U_1^{f_1} \oplus \ldots \oplus U_r^{f_r} \oplus Y$ such that $Y$ has no summand isomorphic to any $U_i$. Then we put $\sum_{i=1}^{r} e_if_i$ arrows from the vertex $V$ to the vertex $W$ in $Q$.\\

\noindent \textbf{Step 4}: Repeat Steps 2-4 for all representative unfactorizable morphisms.\

\begin{remark}
We can index an arrow in $Q$ by a list $(\alpha, V, W, U, s, l)$: this arrow is induced by a representative unfactorizable morphism $\alpha: x\rightarrow y$; it starts from a simple $k\text{Aut} _{\mathcal{C}} (x)$-module $V$ and ends at a simple $k\text{Aut} _{\mathcal{C}} (y)$-module $W$. This arrow is associated with the $s$-th isomorphic copy of $U$ in $V\downarrow_{G_1}^{G}$ and the $l$-th isomorphic copy of $U$ in $W\downarrow_{H_1}^H$, where $U$ is a common summand of $V\downarrow_{G_1}^{G}$, $W\downarrow_{H_1}^{H}$ and $k\uparrow_{G_0}^{G_1}$ under the given identification $(G_1/G_0) \cong (H_1/H_0)$. See Lemma 3.3, Lemma 3.5 and Lemma 3.6 for more details.\\

Although the subspaces $M(V,U,s)$ of $M$ and $N(W,U,l)$ of $N$ depend on the particular decomposition, the quiver we constructed by the above algorithm is independent of it. Indeed, we have:
\begin{equation*}
e_i = \text {dim}_k \text{Hom}_{G_1} (U_i, V\downarrow _{G_1}^G)
\end{equation*}
and
\begin{equation*}
f_i = \text {dim}_k \text{Hom}_{H_1} (U_i, V\downarrow _{H_1}^H)
\end{equation*}
since $U_i$ can be viewed as both a $kG_1$-module and a $kH_1$-module by identifying $G_1/G_0$ with $H_1/H_0$. Clearly $e_i$ and $f_i$ are invariant with respect to different decompositions.
\end{remark}

The following two simple examples illustrate our construction.

\begin{example} Let $\mathcal{C}$ be the following finite EI category where: $G$ and $L$ are cyclic groups of order 2 and 3 respectively; $H$ and $K$ are the symmetric groups on 3 letters; $O_1$ is an $(H,G)$-biset with two morphisms generated by $\alpha$, fixed by the trivial subgroup of $G$ and the proper subgroup of order 3 in $H$; $O_2$ is a $(K,H)$-biset of 6 morphisms generated by $\beta$, permuted regularly by $H$ and $K$; $O_3$ is a $(L,H)$-biset with one morphism $\delta$, fixed by both $H$ and $L$.
\end{example}
\centerline{
\xymatrix{G \ar[r]^{O_1} & H \ar[d]^{O_3} \ar[r]^{O_2} & K\\
& L
}}

The reader may check that $\mathcal{C}$ is in fact a finite free EI category. We have: $kG \cong k \oplus S$, the direct sum of two non-isomorphic one dimensional modules;  $kH \cong kK \cong k \oplus \epsilon \oplus V_2 \oplus V_2$, the direct sum of the trivial module $k$, the sign representation $\epsilon$, and two isomorphic copies of 2-dimensional simple modules; $kL \cong k \oplus \omega \oplus \omega^2$, the direct sum of three pairwise non-isomorphic one dimensional modules. Thus the associated quiver $Q$ has 11 vertices.\

To distinguish different vertices, we mark vertices corresponding to $kG$-modules with $\star$, vertices corresponding to $kH$-modules with $\diamond$, vertices corresponding to $kK$-modules with $\circ$, and vertices corresponding to $kL$-modules with $\bullet$. Now we want to put arrows among them. It is not hard to see that $\mathcal{C}$ has 3 representative unfactorizable morphisms $\alpha$, $\beta$ and $\delta$.\

First, let us analyze $\alpha$, which determines arrows from vertices marked by $\star$ to vertices marked by $\diamond$. Obviously, $G_0=1$, $G_1 = G$; $H_0 = C_3 \unlhd H$, $H_1=H$; and $G_1 /G_0 \cong H_1 /H_0 \cong C_2$. Thus $\alpha$ gives rise to no arrow ending at the vertex $\diamond V_2$ since it is not a simple summand of $k \uparrow_{H_0}^{H}$. All other vertices are summands of $k\uparrow _{G_0}^G$ and $k \uparrow _{H_0}^H$. There is one arrow from $\star k$ to $\diamond k$ since by identifying $G_1/G_0$ and $H_1/H_0$, they are the trivial representations of this quotient group. Similarly, there is an arrow from $\star S$ to $\diamond \epsilon$ since they both are the sign representations of this quotient group.\

We omit the detailed analysis of $\beta$ and $\delta$. Finally, we get the associated quiver $Q$ as below:\\
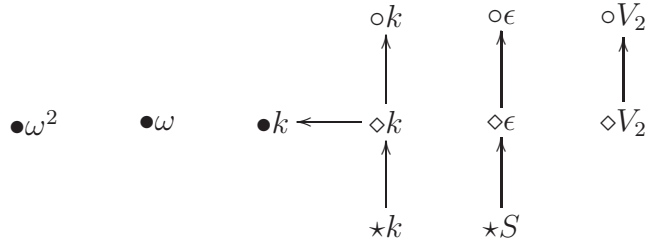
\centerline{
\xymatrix{ & & & \circ k & \circ \epsilon & \circ V_2 \\
\bullet \omega^2 & \bullet \omega & \bullet k & \diamond k \ar[l] \ar[u] & \diamond \epsilon \ar[u] & \diamond V_2 \ar[u] \\
& & & \star k \ar[u] & \star S \ar[u]\\}
}

\begin{example} Let $\mathcal{C}$ be a finite EI category with objects $x$ and $y$; $H=$Aut$_{\mathcal{C}} (y)$ is a copy of the symmetric group $S_3$ on 3 letters; $G=$Aut$_{\mathcal{C}} (x)$ is cyclic of order 2; Hom$_{\mathcal{C}} (x,y)= S_3$ regarded as an $(H,G)$-biset where $H$ acts from the left by multiplication, $G$ acts from the right by multiplication after identifying $G$ with a subgroup $G^{\dagger}$ of $S_3$; Hom$_{\mathcal{C}} (y,x) = \emptyset$.
\end{example}

From the previous example, we find that $Q$ has 5 vertices: $\circ k$ and $\circ S$ corresponding to $x$; $\bullet k$, $\bullet V_2$ and $\bullet \epsilon$ corresponding to $y$. We choose $\alpha = 1 \in S_3$ as the representative unfactorizable morphism and then $G_0 = 1$, $G_1 = G$, $H_0 = 1$, $H_1 = G^{\dagger}$. As before, $kG \cong k \oplus S$, and $kH \cong k \oplus \epsilon \oplus V_2 \oplus V_2$. Moreover, $k\downarrow_{H_1}^{H} \cong k$, $\epsilon \downarrow_{H_1}^{H} \cong S$, $V_2\downarrow_{H_1}^{H} \cong S \oplus k$. Thus the quiver $Q$ is as follows:\\

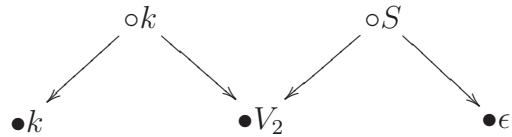
\centerline{
\xymatrix{ & \circ k \ar[dl] \ar[dr] & & \circ S \ar[dl] \ar[dr] \\
\bullet k & & \bullet V_2 & & \bullet \epsilon\\
&
}
}

In the above examples we find that the associated quivers are acyclic. Moreover, if $\mathcal{C}$ is a finite free EI category, the underlying quiver of a finite EI quiver $\hat{Q}$ generating $\mathcal{C}$ is a subquiver of $Q$. This is always true.\

\begin{proposition}
Let $\mathcal{C}$ be a finite EI category for which the endomorphism groups of all objects have orders invertible in $k$. The associated quiver $Q$ of $\mathcal{C}$ is acyclic. Moreover, if $\mathcal{C}$ is a finite free EI category generated by a finite EI quiver $\hat{Q}$, then $Q$ contains a subquiver isomorphic to the underlying quiver of $\hat{Q}$.
\end{proposition}

\begin{proof}
It is obvious that the associated quiver $Q$ of $\mathcal{C}$ contains no loops. Now let $S_1 \rightarrow \ldots \rightarrow S_n=S_1$ be an oriented cycle with at least 2 vertices (therefore $n \geqslant 3$), where $S_i$ is a simple summand of $k$Aut$_{\mathcal{C}} (x_i)$, $x_i \in \text{Ob} (\mathcal{C})$, $1 \leqslant i \leqslant n$. Our construction shows that each arrow is induced by a unfactorizable morphism in $\mathcal{C}$. Therefore we get a string of unfactorizable morphisms
\begin{equation*}
\xymatrix{x_1 \ar[r]^-{\alpha_2} & \ldots \ar[r]^-{\alpha_{n-1}} & x_{n-1} \ar[r]^-{\alpha_n} & x_n=x_1}.
\end{equation*}
Notice that $x_2 \neq x_1$ since $\alpha_2$ is unfactorizable, hence not an automorphism. But this implies that both Hom$_{\mathcal{C}} (x_1, x_2)$ and Hom$_{\mathcal{C}} (x_2, x_1)$ are non-empty, which is a contradiction (see the paragraphs before Definition 2.1), so the first statement is correct.\

Now suppose that $\mathcal{C}$ is a finite free EI category generated by a finite EI quiver $\hat{Q}$. We know that vertices in $\hat{Q}$ are exactly objects in $\mathcal{C}$, and each object $x$ in $\mathcal{C}$ gives a unique vertex $k_x$ in $Q$, the trivial $k \text{Aut} _{\mathcal{C}} (x)$-module. Furthermore, every arrow $x \rightarrow y$ in $\hat{Q}$ corresponds to a unique representative unfactorizable morphism $\alpha: x \rightarrow y$, which in turn determines, by our construction, a unique arrow from $k_x$ to $k_y$ in $Q$. By identifying $x$ with $k_x$, we get the second statement.
\end{proof}

This associated quiver $Q$ of $\mathcal{C}$ is actually the ordinary quiver of $k\mathcal{C}$. To prove this, let us consider the radical of $k\mathcal{C}$.

\begin{proposition}
Let $\mathcal{C}$ be a finite EI category for which the endomorphism groups of all objects have orders invertible in $k$. Then $\rad k\mathcal{C}$ has as a basis the non-isomorphisms in $\mathcal{C}$ and $\rad k\mathcal{C} / \rad^2 k\mathcal{C}$ has as a basis the images of all unfactorizable morphisms in $\mathcal{C}$ respectively.
\end{proposition}

\begin{proof}
Let $\Lambda$ be the subspace of $k\mathcal{C}$ spanned by all non-isomorphisms in $\mathcal{C}$. Clearly, $\Lambda$ is an two-sided ideal of $k\mathcal{C}$ and it has the non-isomorphisms as a basis. Moreover, $k\mathcal{C} / \Lambda \cong \bigoplus_{x \in \text{Ob} (\mathcal{C})} {k\text{End}_{\mathcal{C}} (x)}$. Since the endomorphism groups of all objects in $\mathcal{C}$ have orders invertible in $k$, all groups algebras are semisimple. Thus $k\mathcal{C} / \Lambda$ is semisimple, and $\rad k\mathcal{C} \subseteq \Lambda$. On the other hand, because $\mathcal{C}$ has only finitely many distinct objects, $\Lambda$ is nilpotent, so $\Lambda \subseteq \rad k\mathcal{C}$. This proves $\Lambda = \rad k\mathcal{C}$.\

As an ideal of $k\mathcal{C}$, $\rad^2 k\mathcal{C}$ contains all factorizable morphisms in $\mathcal{C}$, and no unfactorizable morphism. Consequently, $\rad k\mathcal{C} / \rad^2{k\mathcal{C}}$ is spanned by the images of all unfactorizable morphisms in $\mathcal{C}$. Actually, these images form a basis of $\rad k\mathcal{C} / \rad^2{k\mathcal{C}}$.
\end{proof}

Now we restate and prove Theorem 1.1.

\begin{theorem}
Let $\mathcal{C}$ be a finite EI category for which the endomorphism groups of all objects have orders invertible in the field $k$. Then the quiver $Q$ constructed by our algorithm is precisely the ordinary quiver of the category algebra $k\mathcal{C}$. Moreover, $k\mathcal{C}$ has the same ordinary quiver as that of $k \hat{\mathcal{C}}$, the category algebra of the free EI cover $\hat{\mathcal{C}}$ of $\mathcal{C}$.
\end{theorem}
\begin{proof}
Let $Q'$ be the ordinary quiver of $k\mathcal{C}$. We only need to show that there is a bijection $\pi$ from the vertices set of $Q$ to the vertices set of $Q'$, and for every pair of vertices $v$ and $w$ in $Q$, the number of arrows from $v$ to $w$ is the same as that of the arrows from $\pi(v)$ to $\pi(w)$. By Corollary 4.5 of ~\cite{Webb1}, a primitive idempotent of $k\mathcal{C}$ is exactly a primitive idempotent of $k \text{End} _{\mathcal{C}} (x)$ for some object $x$. Moreover, if $e_x$ and $e_y$ are two primitive idempotents of $k\mathcal{C}$ associated with objects $x$ and $y$ respectively, then $k\mathcal{C}e_x \cong k\mathcal{C}e_y$ if and only if $x = y$ (since $\mathcal{C}$ is skeletal) and $k \text{End} _{\mathcal{C}} (x) e_x \cong k \text{End} _{\mathcal{C}} (x) e_y$ (see the paragraph before Proposition 4.3 of ~\cite{Webb1}). Thus the vertices of $Q'$, formed by isomorphism classes of indecomposable projective $k\mathcal{C}$-modules, can be parametrized by isomorphism classes of indecomposable $k \text{End} _{\mathcal{C}} (x)$-modules, where $x$ ranges over the objects in $\mathcal{C}$. But because every $k \text{End} _{\mathcal{C}} (x)$ is semisimple, indecomposable projective $k \text{End} _{\mathcal{C}} (x)$-modules and simple $k \text{End} _{\mathcal{C}} (x)$-modules coincide. Therefore $Q$ and $Q'$ have the same vertices.\

Now let $e_1$ and $e_2$ be two primitive idempotents of $k\mathcal{C}$ corresponding to objects $x$ and $y$ respectively. Let $v = [k\mathcal{C}e_1]$ and $w = [k\mathcal{C}e_2] $ be the corresponding vertices in $Q'$ (or $Q$ since they have the same vertices). We know that the numbers of arrows in $Q'$ from $v$ to $w$ is the dimension of the $k$-space $e_2 (\text {rad} k\mathcal{C} / \text{rad}^2 k\mathcal{C}) e_1$. Let $\alpha_1, \alpha_2, \ldots, \alpha_l$ be all representative unfactorizable morphisms from $x$ to $y$. We verify that
\begin{equation*}
e_2 (\text {rad} k\mathcal{C} / \text{rad}^2 k\mathcal{C}) e_1 \cong \bigoplus _{i=1}^l e_2(kH\alpha_iG)e_1
\end{equation*}
by the previous proposition. Thus we only need to check that each representative unfactorizable morphism from $x$ to $y$ gives the same number of arrows from $v$ to $w$ as that given by our algorithm.\

Take a particular representative unfactorizable $\alpha$ in the list $\alpha_1, \alpha_2, \ldots, \alpha_l$. Let us compute the dimension of $e_2 (kH\alpha G) e_1$. As we mentioned before, $\alpha$ determines groups $H_0 \unlhd H_1 \leqslant H$, $G_0 \unlhd G_1 \leqslant G$ with $H_0 \alpha = \alpha G_0 = \alpha$ and $H_1\alpha = \alpha
G_1$, and $k(G_1/G_0) \cong k(H_1/H_0)$. We identify these two modules and let $\{ u_1, \ldots, u_n \}$ be a list of primitive idempotents of $k(H_1/H_0)$ such that every simple sumand of $k(H_1/H_0)$ is isomorphic to some $k(H_1/H_0)u_i$, and $k(H_1/H_0) u_i \ncong k(H_1/H_0) u_j$ if $i \neq j$ for $ 1 \leqslant i,j \leqslant n$. Notice that all $u_i$ are primitive idempotents of both $kG_1$ and $kH_1$ under the identification. Let $U_i = k(H_1/H_0) u_i$,
\begin{equation*}
kGe_1 \downarrow _{G_1}^G \cong U_1^{a_1} \oplus \ldots \oplus U_n^{a_n} \oplus U
\end{equation*}
and
\begin{equation*}
kHe_2 \downarrow _{H_1}^H \cong U_1^{b_1} \oplus \ldots \oplus U_n^{b_n} \oplus V
\end{equation*}
such that both $U \cong kG_1u$ and $V \cong kH_1v$ have no summand isomorphic to any $U_i$, where $u$ and $v$ are idempotents of $kG$ and $kH$ respectively. Then by using both the right and left module structure of $kH\alpha G$ we have:
\begin{align}
e_2 kH \alpha G e_1 & \cong \text{Hom}_{kH} (kHe_2, kH\alpha Ge_1) \nonumber \\
& \cong \text{Hom}_{kH} ( kHe_2, \text{Hom}_{kG} (e_1kG, kH\alpha G)) \nonumber \\
& \cong \text{Hom}_{kH} ( kHe_2, \text{Hom}_{kG} (e_1kG, \bigoplus_{h \in H/H_1} kh\alpha G)) \nonumber \\
& \cong \text{Hom}_{kH} ( kHe_2, \text{Hom}_{kG} (e_1kG, k\alpha G) \uparrow_{H_1}^H) \nonumber \\
& \cong \text{Hom}_{kH} ( kHe_2, \text{Hom}_{kG} (e_1kG, \bigoplus _{g \in G \backslash G_1} k\alpha gG_1 ) \uparrow_{H_1}^H) \nonumber \\
& \cong \text{Hom}_{kH} ( kHe_2, \text{Hom}_{kG} (e_1kG,  k\alpha G_1 \uparrow_{G_1}^G ) \uparrow_{H_1}^H ) \nonumber \\
& \cong \text{Hom}_{kH} ( kHe_2, \text{Hom}_{kG_1} (e_1kG \downarrow _{G_1}^{G},  k\alpha G_1 ) \uparrow_{H_1}^H) \nonumber \\
& \cong \text{Hom}_{kH_1} ( kHe_2 \downarrow _{H_1}^H, \text{Hom}_{kG_1} (e_1kG \downarrow _{G_1}^{G},  k\alpha G_1 ))
\end{align}

The last two isomorphisms come from the Frobenius Reciprocity. Also notice that $k\alpha G_1 = kH_1 \alpha$ is a $(kH_1, kG_1)$-bimodule, which is isomorphic to $k(G_1/G_0)$ as a right $kG_1$-module and is isomorphic to $k(H_1/H_0)$ as a left $kH_1$-module. Moreover, the decomposition
\begin{equation*}
kGe_1 \downarrow _{G_1}^G \cong \bigoplus _{i=1}^n (kG_1u_i)^{a_i} \oplus kG_1u
\end{equation*}
as a left $kG_1$-module implies the decomposition
\begin{equation*}
e_1kG \downarrow_{G_1}^G \cong \bigoplus_{i=1}^n (u_ikG_1)^{a_i} \oplus ukG_1
\end{equation*}
as a right $kG_1$-module. Thus:
\begin{align*}
\text{Hom} _{kG_1} (e_1kG \downarrow _{G_1}^{G}, k\alpha G_1) & \cong \text{Hom}_ {kG_1} (\bigoplus _{i=1}^n (u_ikG_1)^{a_i} \oplus ukG_1, k\alpha G_1) \\
& \cong \text{Hom} _{kG_1} (\bigoplus _{i=1}^n (u_i kG_1)^{a_i}, k\alpha G_1) \cong \bigoplus _{i=1}^n (k\alpha G_1 u_i)^{a_i} \\
& \cong \bigoplus _{i=1}^n (kH_1\alpha u_i)^{a_i} \cong \bigoplus _{i=1}^n U_i^{a_i}
\end{align*}

Consequently, we have:
\begin{align*}
\text{R.H.S of (4.1)} & \cong \text{Hom} _{kH_1} (kHe_2 \downarrow _{H_1}^H, \bigoplus _{i=1}^n U_i^{a_i}) \\
& \cong \text{Hom} _{kH_1} (\bigoplus _{i=1}^n U_i^{b_i} \oplus V, \bigoplus _{i=1}^n U_i^{a_i}) \\
& \cong \text{Hom} _{kH_1} (\bigoplus _{i=1}^n U_i^{b_i}, \bigoplus _{i=1}^n U_i^{a_i}) \\
& \cong \bigoplus _{i=1}^n ( \text{Hom} _{kH_1} (U_i, U_i) )^{a_ib_i}
\end{align*}

The right hand side of the above identity has dimension $\sum_{i=1}^n a_ib_i$, which is the same as the number of arrows from the vertex $v = [k \mathcal{C} e_1]$ to the vertex $w = [k \mathcal{C} e_2]$ given by $\alpha$ in our construction (see Step 2 and Step 3 of our algorithm and Remark 4.2). This proves the first statement.\

By Proposition 2.9, we know that $\mathcal{C}$ and its free EI cover $\hat{\mathcal{C}}$ have the same objects. Moreover, for each pair of objects $x$ and $y$, we have End$_{\mathcal{C}} (x) = \text{End} _{\hat{\mathcal{C}}} (x)$, End$_{\mathcal{C}} (y) = \text{End} _{\hat{\mathcal{C}}} (y)$, and the unfactorizable morphisms from $x$ to $y$ in $\mathcal{C}$ are the same as those in $\hat{\mathcal{C}}$. But from the above proof we find that these data completely determine their ordinary quivers. Thus $k\hat{\mathcal{C}}$ and $k\mathcal{C}$ have the same ordinary quiver, as we claimed.
\end{proof}

\section{Hereditary Category Algebras}

It is well known that the group algebra $kG$ of a finite group $G$ is hereditary if and only if the order of $G$ is invertible in $k$, and the path algebra $kQ$ of a finite acyclic quiver $Q$ is always hereditary. Since a finite EI category $\mathcal{C}$ can be viewed as a combination of several finite groups and a finite acyclic quiver, it is convincing that a characterization of hereditary category algebras $k\mathcal{C}$ must be related closely to the following conditions: the endomorphism groups of all objects in $\mathcal{C}$ have orders invertible in $k$, and $\mathcal{C}$ is a finite free EI category. These two conditions actually characterize finite EI categories among hereditary category algebras and this is the content of Theorem 1.2.\

Let us state some preliminary results which will be used in the proof of Theorem 1.2.

\begin{lemma}
Let $\alpha_1$ and $\alpha_2$ be two different unfactorizable morphisms in a finite free EI category $\mathcal{C}$. Then as $k\mathcal{C}$-modules, either $k\mathcal{C} \alpha_1 = k\mathcal{C} \alpha_2$, or $k\mathcal{C} \alpha_1 \cap k\mathcal{C} \alpha_2 = 0$.
\end{lemma}

\begin{proof}
For $i = 1,2$, let $B_i = \{\delta \alpha_i \mid \delta \in \text{Mor} (\mathcal{C}), s(\delta) = t(\alpha_i) \}$ be the set of all composites starting with $\alpha_i$, where $s(\delta)$ and $t(\alpha_i)$ are the source of $\delta$ and the target of $\alpha_i$ respectively. Then $B_i$ spans $k\mathcal{C} \alpha_i$. If we can that prove $B_1$ and $B_2$ are either the same, or have empty intersection, then the conclusion follows.\

Suppose that $B_1 \cap B_2$ is nonempty. Then we can find $\beta \in B_1 \cap B_2$ and $\beta$ can be expressed as $\delta_1 \alpha_1 = \delta_2 \alpha_2$. But $\mathcal{C}$ is a finite free EI category and satisfies the UFP, so $\alpha_1$ and $\alpha_2$ have the same source $x$ and the same target $y$. Moreover, there is an automorphism $h \in \text{End} _{\mathcal{C}} (y)$ such that $\alpha_1 = h\alpha_2$ and $\delta_1 = \delta_2 h^{-1}$. In particular, $\alpha_1 \in B_2$ and $\alpha_2 \in B_1$. Thus $B_1 = B_2$, which completes the proof.\
\end{proof}

\begin{lemma}
Let $\alpha: x \rightarrow y$ be an unfactorizable morphism in a finite free EI category for which all endomorphisms groups of objects have orders invertible in $k$. Then the cyclic $k\mathcal{C}$-module $k\mathcal{C} \alpha$ is projective.
\end{lemma}

\begin{proof}
Let $H = \text{End} _{\mathcal{C}} (y)$ and $H_0 = \text{Stab} _{H} (\alpha)$. Since the order of $H$ is invertible in $k$, the subgroup $H_0$ has invertible order in $k$, too. Thus we can define $e = \frac{1}{|H_0|} \sum _{h \in H_0} h \in kH \subseteq k\mathcal{C}$. Since $e$ is an idempotent of $k\mathcal{C}$, it follows that $k\mathcal{C} e$ is a projective $k\mathcal{C}$-module. We will show that $k\mathcal{C} \alpha \cong k\mathcal{C}e$ as $k\mathcal{C}$-modules.\

Define a map $\varphi: k\mathcal{C} \alpha \rightarrow k\mathcal{C} e$ by letting $\varphi (r \alpha) = re$, where $r \in k\mathcal{C}$. We claim that $\varphi$ is a $k\mathcal{C}$-module isomorphism.\

First, we want to show that $\varphi$ is well-defined. That is, if $\sum_{i=1}^m a_i \delta_i \alpha = \sum_{j=1} ^n b_j \beta_j \alpha$ are two different expressions of a vector in $k\mathcal{C} \alpha$, where $a_i, b_j \in k$, $\delta_i, \beta_j \in \text{Mor} (\mathcal{C})$ are morphisms starting at $y$, then $\sum_{i=1}^m a_i \delta_i e = \sum_{j=1}^n b_j \beta_j e$. This is equivalent to saying that if $\sum_{i=1}^m a_i \delta_i \alpha =0$, then $\sum_{i=1}^m a_i \delta_i e=0$, where all $\delta_i$ are pairwise distinct morphisms starting at $y$.\

Those $\delta_i \alpha$ might not be all distinct. By changing the indices if necessary, we can suppose that $\delta_1\alpha = \ldots = \delta_{i_1} \alpha$, $\delta_{i_1+1} \alpha = \ldots = \delta_{i_2} \alpha$, and so on until $\delta_{i_{l-1}+1} \alpha = \ldots = \delta_{i_l} \alpha$, where $\delta_{i_1} \alpha, \delta_{i_2} \alpha, \ldots, \delta_{i_l} \alpha$ are pairwise distinct and $\delta_{i_l} = \delta_m$. From the definition of the category algebra, $\delta_{i_1} \alpha, \ldots, \delta_{i_l}\alpha$ are in fact linearly independent. Thus we have:
\begin{align*}
\sum_{i=1}^m a_i \delta_i\alpha & = (a_1 \delta_1 \alpha + \ldots + a_{i_1}\delta_{i_1} \alpha) + \ldots + (a_{i_{l-1}+1} \delta_{i_{l-1}+1} \alpha + \ldots + a_{i_l} \delta_{i_l} \alpha) \\
& = (a_1 + \ldots + a_{i_1}) \delta_{i_1} \alpha + \ldots + (a_{i_{l-1}+1} + \ldots + a_{i_l}) \delta_{i_l} \alpha
\end{align*}

Notice that $(a_1 + \ldots + a_{i_1}) = \ldots = (a_{i_{l-1}+1} + \ldots + a_{i_l}) = 0$ by the independence of the $\delta_{i_j} \alpha$, $1 \leqslant j \leqslant l$. We want to show that each term $(a_{i_{j-1}+1} + \ldots + a_{i_j}) \delta_{i_j} \alpha$ is sent to 0 by $\varphi$. Observe that since $\delta_s \alpha = \delta_{i_j} \alpha$ for all $1 \leqslant s \leqslant i_j$, by the UFP, there are $h_s \in \text{End} _{\mathcal{C}} (y)$ such that $h_s \alpha= \alpha$ and $\delta_s = \delta_{i_j} h_s^{-1}$. In particular, $h_s \in H_0$. Thus:

\begin{equation*}
\varphi( \sum_{s=1} ^{i_j} a_s \delta_s \alpha) = \sum_{s=1}^{i_j} a_s \delta_s e = \sum_{s=1} ^{i_j} a_s \delta_{i_j} h_s^{-1} e
\end{equation*}

But all $h_s^{-1} \in H_0$ fix $e$, i.e., $h_s^{-1} e=e$, so the right side of the above identity is actually $(\sum_{s=1} ^{i_j} a_s) \delta_{i_j} e = 0$. This shows that $\varphi$ sends each term to 0, and hence is well-defined.\

Now suppose $\varphi( \sum_{i=1}^m a_i \delta_i \alpha ) = \sum_{i=1}^m a_i \delta_i e =0$. Since $e$ fixes $\alpha$, i.e., $e \alpha =\alpha$, we have $\sum_{i=1}^m a_i \delta_i \alpha = (\sum_{i=1}^m a_i \delta_i e) \alpha = 0$. Thus this map is injective.\

It is clear that $\varphi$ is surjective and is a $k\mathcal{C}$-module homomorphism. In conclusion, $\varphi$ is a $k\mathcal{C}$-module isomorphism. The conclusion is proved.
\end{proof}

Let us restate Theorem 1.2 and give a proof here.

\begin{theorem}
Let $\mathcal{C}$ be a finite EI category $\mathcal{C}$. Then $k\mathcal{C}$ is hereditary if and only if $\mathcal{C}$ is a finite free EI category satisfying that the endomorphism groups of all objects have orders invertible in $k$.
\end{theorem}

\begin{proof}
If $\mathcal{C}$ has only one object, the conclusion holds obviously. Without loss of generality, we suppose that $\mathcal{C}$ has more than one object.\

\textbf{The if part.} By Corollary 5.2 on page 17 of ~\cite{Auslander}, it suffices to prove that $\Lambda = \rad k\mathcal{C}$ is a projective $k\mathcal{C}$-module. By Proposition 4.6, $\Lambda$ has as a basis all non-isomorphisms in $\mathcal{C}$. By Proposition 2.6, $\Lambda$ as a $k\mathcal{C}$-module is the sum of all submodules of the form $k\mathcal{C} \alpha$, where $\alpha$ ranges all unfactorizable morphisms in $\mathcal{C}$. By Lemma 5.1, any two of them either coincide, or have a trivial intersection. Therefore, $\Lambda$ is the direct sum of some of these submodules. But all these submodules are projective by Lemma 5.2, so $\Lambda$ is also projective.\\

\textbf{The only if part.} Let $\mathcal{C}$ be a finite EI category such that $k\mathcal{C}$ is hereditary, we first show that all endomorphism groups of objects have orders invertible in $k$. If this is not true, then $\mathcal{C}$ has an object $x$ whose endomorphism group $G= \text{End} _{\mathcal{C} } (x)$ has order not invertible in $k$. Let $k_x$ be the simple $k\mathcal{C}$-module which is $k$ on $x$, and 0 on other objects. Let $\mathcal{P} \rightarrow k_x$ be a minimal $k\mathcal{C}$-projective resolution of $k_x$. By Lemma 5.1.1 of ~\cite{Xu2}, it induces a projective resolution $\mathcal{P} (x) \rightarrow k$ of $k \text{Aut} _{\mathcal{C}} (x)$-modules. Since Aut$_{\mathcal{C}} (x)$ has order not invertible in $k$, this induced projective resolution has infinite length, so $\mathcal{P} \rightarrow k_x$ must be of infinite length. This is impossible since $k\mathcal{C}$ is hereditary. Consequently, all endomorphism groups of objects in $\mathcal{C}$ have orders invertible in $k$.\

Next we prove that $\mathcal{C}$ is a finite free EI category. By Proposition 2.9, there is a full functor $\hat{F}: \hat{\mathcal{C}} \rightarrow \mathcal{C}$, where $\hat{\mathcal{C}}$ is the free EI cover of $\mathcal{C}$. Moreover, $\hat{F}$ is the identity map restricted to objects, isomorphisms and unfactorizable morphisms. By Lemma 2.12, $\hat{\mathcal{C}}$ gives a surjective algebra homomorphism $\psi: k \hat{\mathcal{C}} \rightarrow k\mathcal{C}$. Since $\hat{\mathcal{C}}$ and $\mathcal{C}$ have the same objects, and the same endomorphism group for each object, we know from Proposition 4.3 of ~\cite{Webb1} that every simple $k\mathcal{C}$-module can be viewed as a simple $k \hat{\mathcal{C}}$-module, giving all simple $k \hat{\mathcal{C}}$-modules. Moreover, $k\mathcal{C}$ and $k \hat{\mathcal{C}}$ have the same ordinary quiver (Theorem 4.7). We know that $k \hat{\mathcal{C}}$ is a hereditary algebra by the conclusion we just proved. Therefore, both $k\mathcal{C}$ and $k\hat {\mathcal{C}}$ are hereditary algebras with the same simple modules and the same ordinary quiver. But for a hereditary algebra, these data completely determine the dimension of this algebra. Consequently, $k\mathcal{C}$ and $k \hat{\mathcal{C}}$ have the same dimension, so $\psi$ is an isomorphism and the functor $\hat{F}: \hat{\mathcal{C}} \rightarrow \mathcal{C}$ must be bijective on morphisms. Therefore, $\hat{F}$ is an isomorphism of categories. In conclusion, $\mathcal{C}$ is isomorphic to $\hat{\mathcal{C}}$, so is a finite free EI category.
\end{proof}

\begin{corollary}
The category algebra $k\mathcal{C}$ of a finite EI category $\mathcal{C}$ is hereditary if and only if for all full subcategories $\mathcal{D}$ of $\mathcal{C}$, the category algebras $k\mathcal{D}$ are hereditary.
\end{corollary}

\begin{proof}
One side is trivial since $\mathcal{C}$ is a full subcategory of itself. Now assume that $k\mathcal{C}$ is hereditary. By the previous theorem, $\mathcal{C}$ is a finite free EI category whose endomorphism groups of objects have orders invertible in $k$. If $\mathcal{D}$ is a full subcategory of $\mathcal{C}$, then by Proposition 2.10, $\mathcal{D}$ is also a finite free EI category for which all endomorphism groups of objects obviously have orders invertible in $k$. By the above theorem again, $k\mathcal{D}$ is hereditary.
\end{proof}

\begin{corollary}
Let $\mathcal{C}$ be a finite free EI category for which all endomorphism groups of objects have orders invertible in $k$. Then $\mathcal{C}$ is of finite (tame, resp.) representation type if and only if its ordinary quiver has underlying graph a disjoint union of Dynkin (Euclidean, resp.) diagrams. Otherwise, it has wild representation type.
\end{corollary}

\begin{proof}
The conclusion comes from the previous theorem and the classification of representation types of quivers.
\end{proof}

As an application of this theorem, we assert that the categories shown in Example 4.3 and Example 4.4 have finite representation type over fields $k$ whose characteristic is not 2 or 3.

\section{Application to Representation Types}

To determine the representation type of a finite EI category $\mathcal{C}$ is an interesting but challenging problem. In this section we only consider the finite EI categories $\mathcal{C}$ for which the endomorphism groups of all objects have orders invertible in $k$. Under this hypothesis, we can construct the ordinary quiver $Q$ of $k\mathcal{C}$ according to the algorithm described in section 4. If furthermore $\mathcal{C}$ is a finite free EI category, its representation type is completely determined by Corollary 5.5. Otherwise, $k\mathcal{C}$ is Morita equivalent to $kQ/I$, where $I$ is an nontrivial admissible ideal of $kQ$. Thus $k\mathcal{C}$ is of finite representation type if so is $kQ$, which is precisely Proposition 2.13 since $k\mathcal{C}$ is Morita equivalent to $kQ$.\

It is well known that if a quiver $Q$ has a subquiver of infinite representation type, $Q$ is of infinite representation type as well. This conclusion holds for finite groups. Finite EI categories have a similar property:

\begin{proposition}
A finite EI category $\mathcal{C}$ is of infinite representation type if it has a full subcategory $\mathcal{D}$ of infinite representation type.
\end{proposition}

\begin{proof} Let $M$ be an indecomposable $k\mathcal{D}$-module.  By Theorem 4.4.2 of ~\cite{Xu2}, the induced module $M\uparrow_{\mathcal{D}}^{\mathcal{C}}$ is an indecomposable $k\mathcal{C}$-module. Moreover, $M$ is $\mathcal{D}$-projective, i.e., $M \mid M\uparrow_{\mathcal{D}}^{\mathcal{C}} \downarrow_{\mathcal{D}}^{\mathcal{C}}$. This implies that every indecomposable $k\mathcal{D}$-module $M$ is a direct summand of $\tilde{M} \downarrow_{\mathcal{D}}^{\mathcal{C}}$, where $\tilde{M}$ is an indecomposable $k\mathcal{C}$-module. If $\mathcal{C}$ is of finite representation type, it has only finitely many non-isomorphic indecomposable modules. But the restriction of these modules to $\mathcal{D}$ can produce only finitely many non-isomorphic indecomposable $k\mathcal{D}$-modules, and hence $\mathcal{D}$ is of finite representation type, which is a contradiction!
\end{proof}

Now let us consider full subcategories of $\mathcal{C}$. The subcategories with one object always have finite representation type since their category algebras are precisely semisimple group algebras. They cannot provide us much information about the representation type of $\mathcal{C}$. Without loss of generality we suppose that $\mathcal{C}$ has more than one object. Consider the connected subcategories with two objects. Since they are always finite free EI categories, we can completely determine their representation types by constructing the ordinary quivers. Fortunately, these categories give us many useful details about the representation type of $\mathcal{C}$.\

Let $\mathcal{D}$ be a connected full subcategory formed by two distinct objects $x$ and $y$. From the previous proposition we know that if $\mathcal{D}$ is of infinite representation type, so is $\mathcal{C}$. Let $G = \text{Aut}_{\mathcal{D}} (x)$, $H = \text{Aut}_{\mathcal{D}} (y)$. \

\begin{corollary}
With the above notation, if Hom$_{\mathcal{D}} (x,y)$ has more than one orbit as an $(H, G)$-biset, then $\mathcal{D}$, hence $\mathcal{C}$ are of infinite representation type.
\end{corollary}

\begin{proof}
Notice that every morphism in Hom$_{\mathcal{D}} (x,y)$ is unfactorizable. If it has more than one orbit, then $\mathcal{D}$ has at least two representative unfactorizable morphisms which are not in the same biset orbit. Take two representative unfactorizable morphisms $\alpha \neq \beta$ lying in different biset orbits and consider the ordinary quiver $Q$ of $\mathcal{D}$. It contains two vertices $k_x$ and $k_y$, the trivial representations of $G$ and $H$ respectively. Since both $\alpha$ and $\beta$ give an arrow from $k_x$ to $k_y$ by our algorithm, there are at least two arrows from $k_x$ to $k_y$. Thus $Q$ contains multiple arrows and is of infinite representation type. By Corollary 5.5, $\mathcal{D}$, hence $\mathcal{C}$ are of infinite representation type.
\end{proof}

\begin{remark}
We comment that Proposition 6.1 and Corollary 6.2 are always true over all algebraically closed fields $k$, no matter what the characteristic of $k$ is. We strengthen Corollary 6.2 under the hypothesis that all endomorphism groups have orders invertible in $k$.
\end{remark}

By this corollary, we suppose that Hom$_{\mathcal{D}} (x,y)$ has only one orbit as an $(H, G)$-biset.

\begin{corollary}
With the above notation, if neither $G$ nor $H$ acts transitively on Hom$_{\mathcal{D}} (x,y)$, then $\mathcal{D}$, hence $\mathcal{C}$ are of infinite representation type.
\end{corollary}

\begin{proof}
Since Hom$_{\mathcal{D}} (x,y)$ has only one orbit as an $(H, G)$-biset, it is generated by one morphism, say $\alpha$.  As before, we define $G_0 = \text{Stab}_G(\alpha)$, $H_0 = \text{Stab}_H(\alpha)$, $G_1 = \text{Stab}_{G} (H\alpha)$, $H_1 = \text{Stab}_{H} (\alpha G)$. Because neither $G$ nor $H$ acts transitively on Hom$_{\mathcal{D}} (x,y)$, $G_1$ is a proper subgroup of $G$, and $H_1$ is a proper subgroup of $H$.\

Let us consider the $kG$-module $k\uparrow _{G_1}^G$. We claim it has exactly one isomorphic copy of the trivial $kG$-module $k$ as a direct summand. Indeed, Hom$_{kG} (k, k\uparrow_{G_1}^G) \cong \text{Hom}_{kG_1} (k \downarrow_{G_1}^ G, k) = \text{Hom}_{kG_1} (k, k) \cong k$, which implies the claim. But the dimension of $k \uparrow_{G_1}^{G}$ is strictly bigger than 1 since $G_1$ is a proper subgroup of $G$, so it has another simple summand $S$ not isomorphic to $k$. Again Hom$_{kG_1} (k, S\downarrow_{G_1}^G) \cong \text{Hom}_{kG} (k \uparrow_{G_1}^ G, S) \neq 0$, and therefore $S\downarrow_{G_1}^G$ has a trivial summand $k$. With the same reasoning, we can find a simple summand $T \mid k\uparrow_{H_1}^H$ such that $T$ is not isomorphic to the trivial representation $k$ of $H$, but $T\downarrow_{H_1}^H$ has a trivial summand $k$.\

Consider the ordinary quiver $Q$ of $\mathcal{D}$. It contains at least four vertices: $k_x$ and $k_y$ which are the trivial modules of $kG$ and $kH$ respectively, $S$, and $T$. By our construction, the representative unfactorizable morphism $\alpha$ induces an arrow $k_x \rightarrow k_y$ and at least one arrow $k_x \rightarrow T$ since $k_x \downarrow_{G_1}^G \cong k \mid T\downarrow_{H_1}^H$ under the identification. It also gives at least one arrow from $S$ to $k_y$ and at least one arrow from $S$ to $T$. Thus we get a subquiver of $Q$, pictured below:

\begin{equation*}
\xymatrix{ k_x \ar[r] \ar[dr] & k_y \\
S \ar[r] \ar[ur] & T}
\end{equation*}

It is clear that the above subquiver is of infinite representation type, so $Q$ is of infinite representation type as well. By Corollary 5.5, $\mathcal{D}$, hence $\mathcal{C}$ are of infinite representation type.\
\end{proof}

Now without loss of generality we suppose that $H$ acts transitively on Hom$_{\mathcal{D}} (x, y)$ (otherwise consider the opposite category $\mathcal{C}^{op}$, which has the same representation type as $\mathcal{C}$), so that $G_0$ is a normal subgroup of $G=G_1$.

\begin{corollary}
Let $S$ be a simple summand of $k \uparrow_{H_0}^{H_1}$. If $S \uparrow_{H_1}^{H}$ is not multiplicity free, or has more than 3 summands, then $\mathcal{D}$, hence $\mathcal{C}$ are of infinite representation type.
\end{corollary}

\begin{proof}
Notice that $k \uparrow_{G_0}^{G}$, after it is identified with $k\uparrow_{H_0} ^H$, has a simple summand isomorphic to $S$. Denote this $kG$-module $S'$ to avoid confusion. Let $T$ be a simple summand of $S \uparrow_{H_1}^H$. Since
\begin{equation*}
\text{Hom} _{kH_1} (S, T\downarrow_{H_1}^H) \cong \text{Hom}_{kH} (S\uparrow_{H_1}^H, T) \neq 0,
\end{equation*}
$S$ is a simple summand of $T\downarrow_{H_1}^H$. Moreover, $S \uparrow_{H_1}^H$ has more than one summand isomorphic to $T$ if and only if $T\downarrow_{H_1}^H$ has more than one summand isomorphic to $S$.\

Consider the ordinary quiver $Q$ of $\mathcal{D}$. Both $S'$ and $T$ are vertices in $Q$ since they are a simple $kG$-module and a simple $kH$-module respectively. If $S \uparrow_{H_1}^H$ has more than one summand isomorphic to $T$, then $T\downarrow_{H_1}^H$ has more than one summand isomorphic to $S$. Thus there are multiple arrows from the vertex $S'$ to the vertex $T$, and $Q$ is of infinite representation type.\

Now suppose $S \uparrow_{H_1}^{H}$ has more than 3 summands. If it is not multiplicity free, we are done. Otherwise, all summands are pairwise non-isomorphic. Thus $S\uparrow_{H_1}^H$ has at least four pairwise non-isomorphic simple summands, say $T_1$, $T_2$, $T_3$ and $T_4$. They are different vertices in $Q$. By our construction, there are arrows from the vertex $S'$ to each of $T_i$, $1 \leqslant i \leqslant 4$. As a result, the underlying graph of $Q$ has a component whose underlying graph is not a disjoint union of Dynkin diagrams. Thus $Q$ is of infinite representation type. The conclusion follows from Corollary 5.5.
\end{proof}

We end this section with a finite EI category of infinite representation type.\

\begin{example} Let $\mathcal{C}$ be a finite EI category with: Ob$(\mathcal{C}) = \{x, y\}$, $\text{Aut} _{\mathcal{C}} (x) =1$, $\text{Aut} _{\mathcal{C}} (y)$ is a copy of the symmetric group $S_3$ on three letters, which acts on Hom$_{\mathcal{C}} (x, y) =S_3$ by multiplication from left, Hom$_{\mathcal{C}} (y,x) = \emptyset$.
\end{example}

It is not hard to construct the ordinary quiver of this category, as shown below and it has infinite representation type.

\begin{equation*}
\xymatrix{ & \circ V_2 & \\
\circ k & \bullet k \ar@<0.5ex>[u] \ar@<-0.5ex>[u] \ar[l] \ar[r] & \circ \epsilon}
\end{equation*}

\section{Appendix}

In section 5 we proved that $k\mathcal{C}$ is hereditary if and only if $\mathcal{C}$ is a finite free EI category for which all endomorphism groups of objects have orders invertible in $k$. In this section we will construct a functor $F: k\mathcal{C} \text{-mod} \rightarrow kQ \text{-mod}$ for such $\mathcal{C}$, where $Q$ is the ordinary quiver of $k\mathcal{C}$ constructed by our algorithm in section 4. This functor is faithful, dense and full, and hence induces a Morita equivalence between $k\mathcal{C}$ and $kQ$. This construction of $F$ actually motivated Theorem 1.2, and gives a proof of its if part.\\

\textbf{Definition of $F$ on objects.} We let $R$ be a representation of $\mathcal{C}$ and show how to define a representation $R'$ of its ordinary quiver $Q$. Take a fixed vertex $V$ in $Q$. By our construction, $V$ is a simple $k\text{Aut} _{\mathcal{C}} (x)$-module for an object $x$ in $\mathcal{C}$. Let the homogeneous space $R(x)(V)$ of $V$ in the $k \text{End} _{\mathcal{C}} (x)$-module $R(x)$ be $V^a$, the direct sum of $a$ copies of $V$. We then define $R'(V) = k^a$. In this way we assign a vector space to each vertex in $Q$. Repeating this process, $R'$ assigns a vector space to each vertex in $Q$.\

Now we want to define a linear map for every arrow $V \rightarrow W$ in $Q$. By Remark 4.2, this arrow is indexed by a list $(\alpha, V, W, U, s, l)$ where $\alpha: x \rightarrow y$ is an representative unfactorizable morphism. This morphism uniquely determines and is uniquely determined by the following data (see the notation before Lemma 4.1 and Remark 4.2): \

$\bullet$ the groups $G_0 \unlhd G_1 \leqslant G$ and $H_0 \unlhd H_1 \leqslant H$;\

$\bullet$ a $kG$-module $R(x) = M$ with $M(V) \cong V^a$, and a $kH$-module $R(y) = N$ with $N(W) \cong W^b$;\

$\bullet$ the subspaces $M(V, U, s) \subseteq M$ and $N(W, U, l) \subseteq N$.\

The linear map $\varphi= R(\alpha) : R(x) \rightarrow R(y)$ induces a linear map $\varphi'$, the composite of the following maps, where all inclusions and projections are defined as before:

\begin{equation*}
\xymatrix{\varphi': M(V, U, s) \ar@{^{(}->}[r] & M \ar[r]^{\varphi} & N \ar@{->>}[r] & N(W, U, l).}
\end{equation*}

Notice that $M(V, U, s) \cong U^a$ since $M(V) \cong V^a$, $N(W, U, l) \cong U^b$ since $N(W) \cong W^b$. By Lemma 3.5 and Remark 3.7, the derived map from a particular isomorphic copy of $U$ in $M(V, U, s)$ into a particular isomorphic copy of $U$ in $N(V, U, l)$ is 0, or a $k(G_1/G_0)$-module isomorphism. Both cases give a scalar multiplication $\lambda$ since $U$ is a simple $k(G_1/G_0)$-module. Thus $\varphi'$ gives the following $b \times a$ block matrix $\tilde{B}$, where $I$ is the dim$U \times \text{dim}U$ identity matrix. This matrix, in turn, gives us a $b \times a$ matrix $B$.
\begin{equation*}
\tilde{B}=
\begin{bmatrix}
\lambda_{11}I & \lambda_{21}I & \ldots & \lambda_{a1}I \\
\lambda_{12}I & \lambda_{22}I & \ldots & \lambda_{a2}I \\
\cdots & \cdots & \ldots & \cdots \\
\lambda_{1b}I & \lambda_{2b}I & \ldots & \lambda_{ab}I
\end{bmatrix}
\qquad \leftrightarrow \qquad B=
\begin{bmatrix}
\lambda_{11} & \lambda_{21} & \ldots & \lambda_{a1} \\
\lambda_{12} & \lambda_{22} & \ldots & \lambda_{a2} \\
\cdots & \cdots & \ldots & \cdots \\
\lambda_{1b} & \lambda_{2b} & \ldots & \lambda_{ab}
\end{bmatrix}
\end{equation*}
The matrix $B$ provides the desired linear map $\theta$ from $R'(x) = k^a$ to $R'(y) = k^b$ with respect to the chosen bases.\

Repeating the above process, we can define a linear map for each arrow in $Q$, hence a representation $R'$ of $Q$. Define $F(R)=R'$.\\

\textbf{Definition of $F$ on morphisms.} Let $\pi = \{ \phi_x: x \in \text{Ob} (\mathcal{C}) \}$ be a homomorphism from a $k\mathcal{C}$-module $R_1$ to another $k\mathcal{C}$-module $R_2$. We define a $kQ$-module homomorphism $\pi'$ from $F(R_1)=R_1'$ to $F(R_2)=R_2'$.\

Take an arbitrary vertex $V$ in $Q$. Suppose $R_1'(V) = k^a$ and $R_2'(V) = k^c$. By our construction, $V$ uniquely determines an object $x$ in $\mathcal{C}$, and $M_1 = R_1(x)$ ($M_2 = R_2(x)$, resp.) satisfies $M_1(V) \cong V^a$ ($M_2(V) \cong V^c$, resp.). Let $G= \text{Aut}_{\mathcal{C}} (x)$. The composite map $\phi_V: \xymatrix{M_1(V) \ar@{^{(}->}[r] & M_1 \ar[r]^{\phi_x} & M_2 \ar@{->>}[r] & M_2(V)}$ is a summand of $\phi_x$ since $M_1$ and $M_2$ are semisimple. This summand, with the following matrix representation $\tilde{C}$ (here $I_1$ is the dim$V \times \text{dim} V$ identity matrix), gives a unique $c \times a$ matrix $C$, which defines a linear map $\phi'_V$ from $R_1'(V) = k^a$ to $R_2'(V) = k^c$ with respect to the chosen bases.

\begin{equation*}
\tilde{C}=
\begin{bmatrix}
\epsilon_{11}I_1 & \epsilon_{21}I_1 & \ldots & \epsilon_{a1}I_1 \\
\epsilon_{12}I_1 & \epsilon_{22}I_1 & \ldots & \epsilon_{a2}I_1 \\
\cdots & \cdots & \ldots & \cdots \\
\epsilon_{1c}I_1 & \epsilon_{2c}I_1 & \ldots & \epsilon_{ac}I_1
\end{bmatrix}
\qquad \leftrightarrow \qquad C=
\begin{bmatrix}
\epsilon_{11} & \epsilon_{21} & \ldots & \epsilon_{a1} \\
\epsilon_{12} & \epsilon_{22} & \ldots & \epsilon_{a2} \\
\cdots & \cdots & \ldots & \cdots \\
\epsilon_{1c} & \epsilon_{2c} & \ldots & \epsilon_{ac}
\end{bmatrix}
\end{equation*}

In this way we get $\pi'= \{\phi'_V: R_1'(V) \rightarrow R_2'(V) \mid V \text{ is a vertex in }Q \}$, a family of linear transformations. Define $F(\pi) = \pi'$. Now we need to verify that it is indeed a $kQ$-module homomorphism.\

\begin{remark}
We constructed $\phi'_V$ from a summand $\phi_V$ of $\phi_x$. Clearly, we can recover $\phi_V$ from $\phi_V'$. In particular, $\phi_V=0 \Leftrightarrow \phi_V'=0$.
\end{remark}

Take an arbitrary arrow $V \rightarrow W$ in $Q$ and suppose that $R_1'(V) = k^a$, $R_1'(W) = k^b$, $R_2'(V) = k^c$, $R_2'(W) = k^d$. Let the linear maps $\theta_1$ and $\theta_2$ assigned to this arrow by $R_1'$ and $R_2'$ have the matrix representations $B_1$ and $B_2$, and $\phi_V', \phi_W' \in F(\pi) = \pi'$ have the matrix representations $C_x$ and $C_y$:

\begin{equation*}
B_1=
\begin{bmatrix}
\lambda_{11} & \lambda_{21} & \ldots & \lambda_{a1} \\
\lambda_{12} & \lambda_{22} & \ldots & \lambda_{a2} \\
\cdots & \cdots & \ldots & \cdots \\
\lambda_{1b} & \lambda_{2b} & \ldots & \lambda_{ab}
\end{bmatrix}
\qquad B_2=
\begin{bmatrix}
\mu_{11} & \mu_{21} & \ldots & \mu_{a1} \\
\mu_{12} & \mu_{22} & \ldots & \mu_{a2} \\
\cdots & \cdots & \ldots & \cdots \\
\mu_{1b} & \mu_{2b} & \ldots & \mu_{ab}
\end{bmatrix}
\end{equation*}
\begin{equation*}
C_x=
\begin{bmatrix}
\epsilon_{11} & \epsilon_{21} & \ldots & \epsilon_{a1} \\
\epsilon_{12} & \epsilon_{22} & \ldots & \epsilon_{a2} \\
\cdots & \cdots & \ldots & \cdots \\
\epsilon_{1c} & \epsilon_{2c} & \ldots & \epsilon_{ac}
\end{bmatrix}
\qquad C_y=
\begin{bmatrix}
\eta_{11} & \eta_{21} & \ldots & \eta_{b1} \\
\eta_{12} & \eta_{22} & \ldots & \eta_{b2} \\
\cdots & \cdots & \ldots & \cdots \\
\eta_{1d} & \eta_{2d} & \ldots & \eta_{bd}
\end{bmatrix}
\end{equation*}

As we mentioned before, this arrow is indexed by a list $(\alpha, V, W, U, s, l)$. Let $x$ and $y$ be the source and target of $\alpha$ respectively. Then $V$ and $W$ are a simple $k\text{Aut}_{\mathcal{C}} (x)$-module and a simple $k\text{Aut}_{\mathcal{C}} (y)$-module respectively. Since $\pi = \{ \phi_x: x \in \text{Ob} (\mathcal{C}) \}$ is a $k\mathcal{C}$-module homomorphism from $R_1$ to $R_2$, the following diagram commutes:

\begin{equation}
\xymatrix{
M_1 \ar[r]^{\varphi_1} \ar[d]^{\phi_x} & N_1 \ar[d]^{\phi_y}\\
M_2 \ar[r]^{\varphi_2} & N_2.}
\end{equation}

Obviously, $\phi_x$ sends $M_1(V)$ into $M_2(V)$, and $\phi_y$ sends $N_1(W)$ into $N_2(W)$. By Lemma 4.1, $\phi_x$ sends $M_1(V, U, s)$ into $M_2(V, U, s)$, and $\phi_y$ sends $N_1(W, U, l)$ into $N_2(W, U, l)$. Consequently, the above commutative diagram induces the following commutative diagram, where all inclusions and projections are defined in the usual sense:
\begin{equation*}
\xymatrix{
M_1(V, U, s) \ar@{^{(}->}[r] \ar[d]^{\phi_{V,U,s}} & M_1(V) \ar@{^{(}->}[r] \ar[d]^{\phi_V} & M_1 \ar[r]^{\varphi} \ar[d]^{\phi_x} & N_1 \ar@{->>}[r] \ar[d]^{\phi_y} & N_1(W) \ar@{->>}[r] \ar[d]^{\phi_W} & N_1(W, U, l) \ar[d]^{\phi_{W,U,l}}\\
M_2(V, U, s) \ar@{^{(}->}[r] & M_2(V) \ar@{^{(}->}[r] & M_2 \ar[r]^{\varphi} & N_2 \ar@{->>}[r] & N_2(W) \ar@{->>}[r] & N_2(W, U, l).}
\end{equation*}
In particular, the following diagram commutes:
\begin{equation}
\xymatrix{
M_1(V, U, s) \ar[r]^{\varphi_1'} \ar[d]^{\phi_{V,U,s}} & N_1(W, U, l) \ar[d]^{\phi_{W,U,l}}\\
M_2(V, U, s) \ar[r]^{\varphi_2'} & N_2(W, U, l).}
\end{equation}

But by our definition of the functor $F$, $\varphi'_1$ and $\varphi_2'$ induce $\theta_1$ and $\theta_2$, the maps assigned to this fixed arrow by $R'_1$ and $R'_2$ respectively. They have the following matrix representation, where $I$ is the dim$U \times \text{dim}U$ identity matrix.
\begin{equation*}
\tilde{B_1}=
\begin{bmatrix}
\lambda_{11}I & \lambda_{21}I & \ldots & \lambda_{a1}I \\
\lambda_{12}I & \lambda_{22}I & \ldots & \lambda_{a2}I \\
\cdots & \cdots & \ldots & \cdots \\
\lambda_{1b}I & \lambda_{2b}I & \ldots & \lambda_{ab}I
\end{bmatrix}
\qquad \tilde{B_2}=
\begin{bmatrix}
\mu_{11}I & \mu_{21}I & \ldots & \mu_{a1}I \\
\mu_{12}I & \mu_{22}I & \ldots & \mu_{a2}I \\
\cdots & \cdots & \ldots & \cdots \\
\mu_{1b}I & \mu_{2b}I & \ldots & \mu_{ab}I
\end{bmatrix}
\end{equation*}

Also from the definition of $F$, $\phi_V$ and $\phi_W$ induces $\phi'_V$ and $\phi_W'$ respectively. They have the following matrix representations, where $I_1$ is the dim$V \times \text{dim}V$ identity matrix, $I_2$ is the dim$W \times \text{dim}W$ identity matrix:
\begin{equation*}
\hat{C_x}=
\begin{bmatrix}
\epsilon_{11}I_1 & \epsilon_{21}I_1 & \ldots & \epsilon_{a1}I_1 \\
\epsilon_{12}I_1 & \epsilon_{22}I_1 & \ldots & \epsilon_{a2}I_1 \\
\cdots & \cdots & \ldots & \cdots \\
\epsilon_{1c}I_1 & \epsilon_{2c}I_1 & \ldots & \epsilon_{ac}I_1
\end{bmatrix}
\qquad \hat{C_y}=
\begin{bmatrix}
\eta_{11}I_2 & \eta_{21}I_2 & \ldots & \eta_{b1}I_2 \\
\eta_{12}I_2 & \eta_{22}I_2 & \ldots & \eta_{b2}I_2 \\
\cdots & \cdots & \ldots & \cdots \\
\eta_{1d}I_2 & \eta_{2d}I_2 & \ldots & \eta_{bd}I_2
\end{bmatrix}
\end{equation*}

It is easy to see from the above commutative diagram that $\phi_{V,U,s}$ and $\phi_{W,U,l}$ have the matrix representations $\tilde{C_x}$ and $\tilde{C_y}$, with $I_1$ and $I_2$ replaced by $I$. Since the diagram (7.2) commutes, it must be true: $\tilde{C_y} \tilde{B_1} = \tilde{B_2} \tilde{C_x}$. But:
\begin{equation}
C_yB_1 = B_2C_x \text{ if and only if } \tilde{C}_y \tilde{B}_1 = \tilde{B}_2 \tilde{C_x},
\end{equation}
that is: $\phi_W' \theta_1 = \theta_2 \phi_V'$, and the following diagram commutes. Since this arrow is arbitrarily chosen, we know $\pi'$ is indeed a $kQ$-module homomorphism.
\begin{equation}
\xymatrix{
R_1'(V) = k^a \ar[r]^{\theta_1} \ar[d]^{\phi'_V} & R_1'(W) = k^b \ar[d]^{\phi'_W}\\
R_2'(V) = k^c \ar[r]^{\theta_2} & R_2'(W) = k^d}
\end{equation}

\begin{remark}
The above proof actually implies that diagram (7.4) commutes if and only if diagram (7.2) commutes.
\end{remark}

Therefore, $F$ maps a $k\mathcal{C}$-module homomorphism to a $kQ$-module homomorphism. $F$ also preserves the homomorphism composition since matrix product preserves composition. Thus $F$ is indeed a functor from $\text{Rep}_k\mathcal{C}$ to $\text{Rep}_kQ$.\\

Let us use Example 4.4 to show how a representation of $\mathcal{C}$ gives a representation of the associated quiver $Q$ and vice versa.\

\begin{example} Let $R$ be a representation of the finite EI category $\mathcal{C}$ shown in Example 4.4 with: $R(x) = \langle v_1 \rangle \oplus \langle v_2 \rangle \oplus \langle v_3 \rangle \cong k \oplus k \oplus S$; $R(y) = \langle w_1 \rangle \oplus \langle w_2 \rangle \oplus \langle w_3, w_4 \rangle \oplus \langle w_5, w_6 \rangle \cong k \oplus \epsilon \oplus V_2 \oplus V_2$. Moreover, since $V_2 \downarrow_{H_1}^H \cong k \oplus S$, we assume that $w_3$ and $w_5$ both generate the trivial submodules on restriction to $H_1$, and $w_4$ and $w_6$ both generates submodules isomorphic to $S$ on restriction to $H_1$.
\end{example}

We already know from Example 4.4 that $\alpha = 1 \in S_3$ can be chosen as the unique representative unfactorizable morphism in $\mathcal{C}$. Let $\varphi = R(\alpha)$. Then $\varphi(v_1)$ and $\varphi(v_2)$ are in the subspace of $R(y)$ generated by $w_1, w_3, w_5$, and $\varphi(v_3)$ lies in the subspace generated by $w_2, w_4, w_6$. Thus $\varphi$ has the following matrix representation:

\begin{equation*}
\begin{bmatrix}
\lambda_{11} & \lambda_{21} & 0 \\
0 & 0 & \lambda_{32} \\
\lambda_{13} & \lambda_{23} & 0\\
0 & 0 & \lambda_{34} \\
\lambda_{15} & \lambda_{25} & 0\\
0 & 0 & \lambda_{36} \\
\end{bmatrix}
\end{equation*}

The induced representation $R'$ of $Q$ is described below, where $M_1 = \begin{bmatrix} \lambda_{11} & \lambda_{21} \end{bmatrix}$, $M_2 = \begin{bmatrix} \lambda_{13} & \lambda_{23} \\ \lambda_{15} & \lambda_{25}\end{bmatrix}$, $M_3 = \begin{bmatrix} \lambda_{34} \\ \lambda_{36} \end{bmatrix}$, and $M_4 =\begin{bmatrix} \lambda_{32} \end{bmatrix}$.

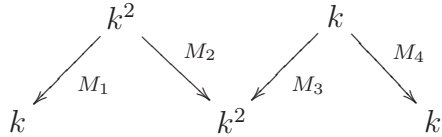
\begin{equation*}
\xymatrix{ & k^2 \ar[dl]^{M_1} \ar[dr]^{M_2} & & k \ar[dl]^{M_3} \ar[dr]^{M_4} \\
k & & k^2 & & k
}
\end{equation*}

The reader can easily recover $R$ from $R'$.

\begin{proposition}
The above functor $F$ we constructed is full, faithful and dense.
\end{proposition}

We give a lemma which will be used in the proof the this proposition.

\begin{lemma}
Let $V = \bigoplus_ {i=1}^m V_i$, $V' = \bigoplus_ {i=1}^m V'_i$, $W = \bigoplus_{i=1}^n W_i$ and $W' = \bigoplus_{i=1}^n W'_i$ be vector spaces. Let $\theta: V \rightarrow V'$, $\varphi: V \rightarrow W$, $\psi: W \rightarrow W'$ and $\varphi': V' \rightarrow W'$  be linear transformations of the following forms. Then the diagram $D$ commutes if and only if all diagrams $D_{ij}$ commute.\\
\begin{equation*}
\theta =
\begin{bmatrix}
\theta_1 & 0 & \ldots & 0 \\
0 & \theta_2 & \ldots & 0 \\
\vdots & \vdots & \ddots & \vdots \\
0 & 0 & \ldots & \theta_m
\end{bmatrix}
\qquad \psi =
\begin{bmatrix}
\psi_1 & 0 & \ldots & 0 \\
0 & \psi_2 & \ldots & 0 \\
\vdots & \vdots & \ddots & \vdots \\
0 & 0 & \ldots & \psi_n
\end{bmatrix}
\end{equation*}
\begin{equation*}
\varphi =
\begin{bmatrix}
\varphi_{11} & \ldots & \varphi_{m1} \\
\varphi_{12} & \ldots & \varphi_{m2} \\
\vdots & \ddots & \vdots \\
\varphi_{1n} & \ldots & \varphi_{mn}
\end{bmatrix}
\qquad \varphi' =
\begin{bmatrix}
\varphi'_{11} & \ldots & \varphi'_{m1} \\
\varphi'_{12} & \ldots & \varphi'_{m2} \\
\vdots & \ddots & \vdots \\
\varphi'_{1n} & \ldots & \varphi'_{mn}
\end{bmatrix}
\end{equation*}
\begin{equation*}
D = \xymatrix{ V \ar[r]^{\varphi} \ar[d]^{\theta} & W \ar[d]^{\psi}\\
V' \ar[r]^{\varphi'} & W'
}
\qquad
D_{ij} = \xymatrix{ V_i \ar[r]^{\varphi_{ij}} \ar[d]^{\theta_i} & W_j \ar[d]^{\psi_j}\\
V'_i \ar[r]^{\varphi'_{ij}} & W'_j
}
\end{equation*}
\end{lemma}

\begin{proof}
This is just a calculation of the matrix product:
\begin{equation*}
\begin{bmatrix}
\psi_1 & 0 & \ldots & 0 \\
0 & \psi_2 & \ldots & 0 \\
\vdots & \vdots & \ddots & \vdots \\
0 & 0 & \ldots & \psi_n
\end{bmatrix}
\begin{bmatrix}
\varphi_{11} & \ldots & \varphi_{m1} \\
\varphi_{12} & \ldots & \varphi_{m2} \\
\vdots & \ddots & \vdots \\
\varphi_{1n} & \ldots & \varphi_{mn}
\end{bmatrix}
=
\begin{bmatrix}
\varphi'_{11} & \ldots & \varphi'_{m1} \\
\varphi'_{12} & \ldots & \varphi'_{m2} \\
\vdots & \ddots & \vdots \\
\varphi'_{1n} & \ldots & \varphi'_{mn}
\end{bmatrix}
\begin{bmatrix}
\theta_1 & 0 & \ldots & 0 \\
0 & \theta_2 & \ldots & 0 \\
\vdots & \vdots & \ddots & \vdots \\
0 & 0 & \ldots & \theta_m
\end{bmatrix}
\end{equation*}
if and only if $\psi_j \varphi_{ij} = \varphi_{ij}' \theta_i$. That is, $D$ commutes if and only if all $D_{ij}$ commute.
\end{proof}

Now let us prove the previous proposition.

\begin{proof}
\textbf{$F$ is dense (or essentially surjective).} First, we show the algorithm giving a representation of $Q$ from a representation of $\mathcal{C}$ is invertible. That is: given a representation $R'$ of $Q$, we can define a representation $R$ of $\mathcal{C}$ such that $F(R) \cong R'$. By Proposition 3.1, it suffices to define a rule $R$ which assigns a $k \text{Aut}_{\mathcal{C}} (x)$-module to each object $x$, a linear transformation to each unfactorizable morphism, so that $R$ restricted to $\mathcal{D}_{\alpha}$ is a representation of $\mathcal{D}_{\alpha}$ for every representative unfactorizable morphism $\alpha$ (see the definition of $\mathcal{D}_{\alpha}$ in section 3).\

An object $x$ gives a family of vertices $\{V_1, \ldots, V_m\}$ in $Q$. If $R'(V_i) = k^{a_i}$, then we define $R(x) = M = V_1^{a_1} \oplus \ldots \oplus V_m^{a_m}$. This definition gives a $k \text{Aut} _{\mathcal{C}} (x)$ -module for every object $x$ in $\mathcal{C}$. Now we construct a linear map $R(\alpha): M=R(x) \rightarrow N=R(y)$ for each representative unfactorizable morphism $\alpha$. As we mentioned before, $M$ as a vector space is the direct sum of all subspaces of the form $M(V, U, s)$, and $N$ as a vector space is the direct sum of all subspaces of the form $N(W, T, l)$ (see the paragraph before Lemma 4.1). Thus it is enough to define linear maps from $M(V, U, s)$ into $N(W, T, l)$.\

The morphism $\alpha$ determines groups $G_0 \unlhd G_1 \leqslant G$. By Remark 3.7, the derived map $\varphi': M(V,U,s) \rightarrow N(W,T,l)$ by $\varphi$ should be a $k(G_1/G_0)$-module homomorphism if $U \cong T \mid k\uparrow_{G_0}^{G_1}$ under the given identification $G_1/G_0 \cong H_1/H_0$; and 0 otherwise. If it is a $k(G_1/G_0)$-module homomorphism, then $U \cong T$ and the list $(\alpha, V, W, U, s, l)$ uniquely determines an arrow from the vertex $V$ to the vertex $W$. Let $B = (\lambda_{ji})_{i=1,\ldots, b}^{j=1, \ldots, a}$ be the matrix representation of the linear map assigned to this arrow by $R'$. It provides a unique block matrix $\tilde{B} = (\lambda_{ji}I)_{i=1,\ldots, b}^{j=1, \ldots, a}$, where $I$ is the dim$U \times \text{dim} U$ identity matrix. This block matrix $\tilde{B}$ gives a linear map from $M(V, U, s)$ to $N(W, T, l)$ under the chosen decomposition of $V \downarrow _{G_1}^G$. All these maps determine a linear map from $M$ to $N$, which can be defined as $R(\alpha)$. Repeating this process, we can define a linear map for each representative unfactorizable morphism.\

By Lemma 3.6, the rule $R$ we just defined is a representation of $\mathcal{D}_{\alpha}$ while restricted to each $\mathcal{D} _{\alpha}$. By Proposition 3.1, $R$ gives a representation of $\mathcal{C}$. By abuse of notation, we denote it by $R$. It is direct to verify that $F(R)$ is isomorphic to $R'$. Thus $F$ is dense.\\

\textbf{$F$ is faithful.} Let $R_1$ and $R_2$ be two $k\mathcal{C}$-modules and $\pi = \{\phi_x \mid x \in \text{Ob} (\mathcal{C}) \}$ be a $k\mathcal{C}$-module homomorphism from $R_1$ to $R_2$. Let $V$ be a vertex in $Q$ and $x$ be the object in $\mathcal{C}$ corresponding to $V$. Let $\phi_x: M_1=R_1(x) \rightarrow M_2=R_2(x)$ be the $k\text{Aut}_{\mathcal{C}}(x)$-module homomorphism in $\pi$, which is the direct sum of $\phi_V: M_1(V) \rightarrow M_2(V)$ for all simple summands $V$ of $M_1$. If $F(\pi) = \{\phi_V' : V \text{ is a vertex in } Q\}$ is 0, then in particular $\phi'_V=0$ and therefore $\phi_V=0$ (see Remark 7.1). Since $V$ is an arbitrary summand of $M_1$, we have $\phi_x =0$. Consequently, $\pi$ is 0.\\

\textbf{$F$ is full.} Let $\pi' = \{\phi_V' : V \text{ is a vertex in } Q\}$ be a $kQ$-module homomorphism from $F(R_1) = R_1'$ to $F(R_2) = R_2'$. We can recover a $k\mathcal{C}$-module homomorphism $\pi = \{\phi_x \mid x \in \text{Ob} (\mathcal{C})\}$ from $\pi'$ such that $F(\pi) = \pi'$. Indeed, to define $\phi_x: M_1 = R_1(x) \rightarrow M_2=R_2(x)$ for a fixed object $x$, it is enough to define $\phi_V: M_1(V) \rightarrow M_2(V)$, where $V$ is a simple summand of $M_1$. According to Remark 7.1, $\phi_V$ could be recovered from $\phi'_{V} \in \pi'$. In this way we get from $\pi'$ a family of linear transformations $\pi =\{\phi_x \mid x \in \text{Ob} (\mathcal{C}) \}$. Clearly, $F(\pi)$ and $\pi'$ have the same matrix representation, i.e., $F(\pi) = \pi$ under the chosen decompositions and the chosen bases. The only thing which remains is to show that $\pi$ is indeed a $k\mathcal{C}$-module homomorphism. By Proposition 3.2, we only need to verify that diagram (7.2) commutes for every representative unfactorizable morphism $\alpha: x \rightarrow y$.\

The $k \text{Aut} _{\mathcal{C}} (x)$-module $M_1$ ($M_2$, resp.) is a direct sum of subspaces of the form $M_1(V, U, s)$ ($M_2(V,U,s)$, resp.). The $k \text{Aut} _{\mathcal{C}} (y)$-module $N_1$ ($N_2$) is a direct sum of subspace of the form $N_1(W, T, l)$ ($N_2(W,T,l)$, resp.). By Lemma 7.1, it is enough to check that for an arbitrary $M_i(V, U, s)$ and an arbitrary $N_i(W, T, l)$, $i=1,2$, diagram (7.2) commutes. The morphism $\alpha$ determines groups $G_0 \unlhd G_1 \leqslant G$. Again by Lemma 3.5 and Remark 3.7, the maps $\varphi'_i: M_i(V,U,s) \rightarrow N_i(W,T,l)$ derived from $\varphi_i$, $i=1, 2$, are 0 or $k(G_1/G_0)$-module homomorphisms under the given identification. If they are 0, the diagram commutes trivially. Otherwise, $U \cong T \mid k\uparrow _{G_0}^{G_1}$, and the list $(\alpha, V, W, U, s, l)$ determines an arrow. The linear maps assigned to this arrow by $R_1'$ and $R_2'$, combined with $\phi_V'$ and $\phi_W'$ give the commutative diagram (7.4). By Remark 7.2 the diagram (7.2) also commutes. Thus $\pi$ is indeed a $k\mathcal{C}$-module homomorphism and $F(\pi) =\pi'$. That is, $F$ is full.\\

We proved that $F$ is dense, faithful and full. By Theorem 1 on page 91 of ~\cite{MacLane} $F$ gives rise to a Morita equivalence between $kQ$ and $k\mathcal{C}$. This finishes the proof.
\end{proof}

\end{document}